\documentclass{amsart}
\usepackage{amssymb,mathtools}
\usepackage{stmaryrd}
\usepackage[british]{babel}
\usepackage{enumitem}
\usepackage[latin1]{inputenc}
\usepackage{url}
\usepackage{hyperref}
\usepackage{turnstile}

\usepackage{bussproofs}

\newtheorem{theorem}{Theorem}
\numberwithin{theorem}{section}
\newtheorem{corollary}[theorem]{Corollary}
\newtheorem{lemma}[theorem]{Lemma}
\newtheorem{proposition}[theorem]{Proposition}
\theoremstyle{definition}
\newtheorem{definition}[theorem]{Definition}
\newtheorem{remark}[theorem]{Remark}
\newtheorem{example}[theorem]{Example}
\newtheorem{exercise}[theorem]{Exercise}

\newtheorem{assumption}[theorem]{Standing Assumption}

\newcommand{\rk}{\operatorname{rk}}
\newcommand{\tl}{\operatorname{\lhd}}
\newcommand{\prog}{\operatorname{\mathsf{Prog}}}
\newcommand{\ti}{\operatorname{\mathsf{TI}}}
\newcommand{\pa}{\mathsf{PA}}
\newcommand{\aca}{\mathsf{ACA}}
\newcommand{\bad}{\operatorname{Bad}}

\title[A First Course on Ordinal Analysis]{Unprovability in Mathematics:\\ A First Course on Ordinal Analysis}
\author{Anton Freund}

\address{Anton Freund, Department of Mathematics, Technical University of Darmstadt, Schloss\-garten\-str.~7, 64289~Darmstadt, Germany}
\email{freund@mathematik.tu-darmstadt.de}

\begin{document}

\begin{abstract}
These are the lecture notes of an introductory course on ordinal analysis. Our selection of topics is guided by the aim to give a complete and direct proof of a mathematical independence result: Kruskal's theorem for binary trees is unprovable in conservative extensions of Peano arithmetic (note that much stronger results of this type are due to Harvey Friedman). Concerning prerequisites, we assume a solid introduction to mathematical logic but no specialized knowledge of proof theory. The material in these notes is intended for 12~lectures and 6~exercise sessions of 90~minutes each.
\end{abstract}

\keywords{Lecture notes, ordinal analysis, concrete incompleteness, Kruskal's tree theorem}
\subjclass[2020]{03-01, 03B30, 03F05, 03F15, 03F35, 03F40}

\maketitle

\section{Introduction}

The incompleteness theorems of Kurt~G\"odel tell us: for any ``reasonable" axiom system for the foundations of mathematics, there are statements that can neither be proved nor refuted (see e.\,g.~\cite{smorynski-incompleteness}). It seems natural to ask for examples of such statements. In particular, one may be interested in known mathematical theorems that are unprovable in some relevant axiom system (while the known proof can be formalized in a stronger system). The present course will culminate in an example of this type. Our example concerns a certain collection $\mathcal B$ of finite binary trees. For $s,t\in\mathcal B$ we write $s\leq_{\mathcal B} t$ if there is an infimum-preserving embedding of $s$ into~$t$  (see Section~\ref{sect:kruskal} for details). The following does, in fact, hold for arbitrary finite trees (rather than just for binary ones). For a proof we refer to Exercise~\ref{ex:Kruskal} below.

\begin{theorem}[{Kruskal~\cite{kruskal60}}]\label{thm:kruskal-binary}
For any infinite sequence $t_0,t_1,\ldots\subseteq\mathcal B$ there are indices $i<j$ with $t_i\leq_{\mathcal B} t_j$.
\end{theorem}

Due to the quantification over infinite sequences, Kruskal's theorem cannot be expressed in the language of Peano arithmetic, at least not without reserve. On the other hand, it is readily expressed in conservative extensions of Peano arithmetic. We will present two such extensions, denoted by $\pa[X]$ and $\aca_0$, respectively (see Sections~\ref{sect:inf-derivs} and~\ref{sect:kruskal}). The following result is the focal point of our lecture course. It follows from work of Gerhard Gentzen (see in particular~\cite{gentzen36,gentzen43}) and Dick de Jongh (unpublished, see~\cite{schmidt75}). Stronger results of the same type are due to Diana Schmidt~\cite{schmidt-habil-new} and, in particular, to Harvey Friedman (see~\cite{simpson85}).

\begin{theorem}\label{thm:kruskal-unprovable}
Kruskal's theorem for binary trees is unprovable in conservative extensions of Peano arithmetic.
\end{theorem}

Let us point out that the statement of the theorem is not fully precise, because it does not specify which formalizations of Kruskal's theorem and which conservative extensions can be admitted. Strictly speaking, we will only prove the theorem for one `natural' formalization and the two conservative extensions $\pa[X]$ and $\aca_0$ that were mentioned above. At the same time, our general and somewhat informal statement of Theorem~\ref{thm:kruskal-unprovable} can be justified by related (but slightly more difficult) results: In any `reasonable' conservative extension, any `reasonable' formalization of Kruskal's theorem implies the consistency of Peano arithmetic, so that we get unprovability by G\"odel's theorem (see e.\,g.~\cite{freund-kruskal-scheme} for details). Furthermore, Harvey Friedman has formulated a ``miniaturized" version of Kruskal's theorem: the statement of this miniaturized Kruskal theorem can be formalized in Peano arithmetic itself, but the latter does not prove it (see again~\cite{simpson85}).

To avoid misunderstanding, we stress that the proof of Kruskal's theorem can be formalized in stronger axiom systems. For the case of binary trees, it suffices to go slightly beyond Peano arithmetic (by the aforementioned result of de Jongh, cf.~Remark~\ref{rmk:reification} below). Kruskal's theorem for arbitrary finite trees is unprovable in theories that are much stronger than Peano arithmetic (see~\cite{simpson85,rathjen-weiermann-kruskal}). The analogue of Kruskal's theorem for graphs (rather than just trees) is the graph minor theorem of Neil Robertson and Paul Seymour~\cite{robertson-seymour-gm}. This theorem is one of the most celebrated results in modern combinatorics, with important implications for computer science. It also leads to one of the most impressive unprovability results in mathematical logic~\cite{friedman-robertson-seymour}. In the present course, we focus on unprovability in Peano arithmetic, because this is much simpler and already very interesting.

Let us compare Theorem~\ref{thm:kruskal-unprovable} with a rather different independence result: Every bounded increasing sequence $(q_n)$ of rationals has a real number~$q$ as limit. However, it can happen that the sequence $(q_n)$ is computable while $q$ cannot be computed with prescribed precision (one will then call $(q_n)$ a Specker sequence). One can deduce that some relevant axiom systems (weaker than Peano arithmetic; see e.\,g.~\cite{simpson09}) cannot prove that all bounded increasing sequences in the reals converge. Crucially, this statement asserts the existence of an infinite object, namely, a real limit given by an infinite Cauchy sequence (or some other representation). It is the complexity of this infinite object that leads to the independence result. In other words, we use the notion of computability as an invariant that classifies (countably) infinite sets. This approach cannot be used to prove ``our" Theorem~\ref{thm:kruskal-unprovable}, since Kruskal's theorem does not assert the existence of any infinite object. Indeed, the theorem is the universal closure (``for any infinite sequence") of a statement that is concerned with finite objects only. In a sense, this makes the unprovability result more surprising and harder to show. The focus on finite objects (in particular in the miniaturized version that was mentioned above) is also relevant with respect to Hilbert's programme in the foundations of mathematics (see e.\,g.~\cite{zach16,rathjen-sieg-stanford}).

In the previous paragraph, we have explained that Theorem~\ref{thm:kruskal-unprovable} is a challenge because Kruskal's theorem does not assert the existence of infinite sets. The present course is an introduction to a powerful method that addresses this challenge: \mbox{ordinal analysis} (which is a subfield of proof theory and hence of mathematical logic). A central idea is to study infinite generalizations of proofs. While we cannot write down infinite proofs in practice, they provide an important technical tool. In particular, the ordinal rank of an infinite proof tree can serve as a bound on the statement that is being proved. The course will work out these ideas in great detail. Our selection of topics is guided by the aim to prove Theorem~\ref{thm:kruskal-unprovable}. Nevertheless, some other important results are treated along the way, so that the course provides a reasonably general introduction to ordinal analysis. Specifically, we will discuss sequent calculus (Tait style), cut elimination, and the ordinal analysis of Peano arithmetic (originally due to Gerhard Gentzen~\cite{gentzen36,gentzen43}). Some further topics are briefly discussed in the concluding Section~\ref{sect:conclusion}. The ordinal analysis of considerably stronger axiom systems is treated in a second course~\cite{second-course} by the present author.

The main sources for this course are a survey by Helmut Schwichtenberg~\cite{schwichtenberg77} and lecture notes by Wilfried Buchholz~\cite{buchholz-lecture-03}. As a non-exhaustive list of additional resources, we mention the survey articles by Michael Rathjen~\cite{rathjen-realm,rathjen-icm} and by Rathjen and Wilfried Sieg~\cite{rathjen-sieg-stanford}, as well as the textbooks by Kurt Sch\"utte~\cite{schuette77}, Gaisi Takeuti~\cite{takeuti-proof-theory}, Jean-Yves Girard~\cite{girard87}, Wolfram Pohlers~\cite{pohlers-proof-theory}, and Helmut Schwichtenberg and Stanley Wainer~\cite{schwichtenberg-wainer}.

\section{Sequent calculus for predicate logic}\label{sect:sequent-calc-PL}

In ordinal analysis and other parts of proof theory, we work directly with proofs as syntactic objects. For this reason, it is important to choose a proof system with convenient technical properties. A canonical choice for ordinal analysis is some variant of sequent calculus. In this section we present a variant due to William Tait (``Tait-style"; cf.~\cite{tait-style}), which is streamlined for classical logic. For other variants of sequent calculus (in particular ones that are suitable for intuitionistic logic), we refer to the textbook by Anne Troelstra and Helmut Schwichtenberg~\cite{schwichtenberg-troelstra}.

Fix some signature for first order logic. Equality has no special status but may, of course, be included as a relation symbol. Terms are defined in the usual way. Formulas are built up from literals (negated and unnegated prime formulas), using the connectives $\land,\lor$ and the quantifiers $\forall,\exists$. In other words, we only admit formulas in negation normal form. Negation becomes a defined operation on formulas, which employs de Morgan's rules and deletes double negations:

\begin{definition}\label{def:NNF}
The negation $\sim\varphi$ of a formula $\varphi$ is recursively defined by
\begin{alignat*}{5}
\sim P\mathbf t&:=\neg P\mathbf t,\quad&\sim(\varphi\land\psi)&:=(\sim\varphi)\lor(\sim\psi),\quad&\sim\forall x.\varphi&:=\exists x.\sim\varphi,\\
\sim(\neg P\mathbf t)&:=P\mathbf t,\quad&\sim(\varphi\lor\psi)&:=(\sim\varphi)\land(\sim\psi),\quad&\sim\exists x.\varphi&:=\forall x.\sim\varphi.
\end{alignat*}
We will also write $\neg\varphi$ for $\sim\varphi$ and $\varphi\to\psi$ for $(\sim\varphi)\lor\psi$.
\end{definition}

A straightforward induction shows that $\neg\neg\varphi$ and $\varphi$ are the same formula. This explains why our approach is suitable for classical logic only.

We agree that a \emph{sequent} is a finite set of formulas (in other settings one may wish to consider multisets or sequences, to control the multiplicity or order of formulas). Sequents are denoted by uppercase greek letters, in particular by $\Gamma$ and $\Delta$. In the context of sequents, a comma will often denote a union, and set braces $\{\cdot\}$ are frequently omitted: for example we write $\Delta,\Gamma,\varphi$ instead of $\Delta\cup\Gamma\cup\{\varphi\}$. Intuitively, a sequent $\Gamma=\varphi_1,\ldots,\varphi_n$ denotes the disjunction $\bigvee\Gamma=\varphi_1\lor\ldots\lor\varphi_n$ of the formulas that it contains. In the following, we write $[x/t]$ to denote substitution of $x$ by~$t$ (with suitable renaming of bound variables).

\begin{definition}\label{def:seq-calc}
The rules of (Tait style) sequent calculus (for predicate logic) are
\vspace*{.1\baselineskip}
\begin{gather*}
\AxiomC{}
\RightLabel{($\theta$ a prime formula),}
\UnaryInfC{$\Delta,\theta,\neg\theta$}
\DisplayProof
\qquad
\AxiomC{$\Gamma,\varphi_0$}
\AxiomC{$\Gamma,\varphi_1$}
\RightLabel{,}
\BinaryInfC{$\Delta,\Gamma,\varphi_0\land\varphi_1$}
\DisplayProof
\qquad
\AxiomC{$\Gamma,\varphi_i$}
\RightLabel{,}
\UnaryInfC{$\Delta,\Gamma,\varphi_0\lor\varphi_1$}
\DisplayProof\\[2ex]
\AxiomC{$\Gamma,\varphi[x/y]$}
\RightLabel{($y$ not free in $\Gamma,\forall x.\varphi$),}
\UnaryInfC{$\Delta,\Gamma,\forall x.\varphi$}
\DisplayProof
\qquad
\AxiomC{$\Gamma,\varphi[x/t]$}
\RightLabel{,}
\UnaryInfC{$\Delta,\Gamma,\exists x.\varphi$}
\DisplayProof
\qquad
\AxiomC{$\Gamma,\varphi$}
\AxiomC{$\Gamma,\neg\varphi$}
\RightLabel{.}
\BinaryInfC{$\Delta,\Gamma$}
\DisplayProof
\end{gather*}
\vspace*{.25\baselineskip}

\noindent A derivation (of $\Gamma$) is a finite tree that is built according to these rules (with $\Gamma$ at the root). To denote that $\Gamma$ has a derivation, we write $\textsf{PL}\vdash\Gamma$ (where $\textsf{PL}$ is omitted when the context makes it clear that we work in predicate logic). The first and last rule above are called \emph{axiom} and \emph{cut}. 
\end{definition}

Note that the sequents in the rules need not be disjoint. This means, e.\,g., that
\begin{prooftree}
\AxiomC{$\varphi_0\lor\varphi_1,\varphi_0$}
\UnaryInfC{$\varphi_0\lor\varphi_1$}
\end{prooftree}
is a legitimate instance of the rule for disjunction (as $\Gamma,\varphi_0\lor\varphi_1$ and $\varphi_0\lor\varphi_1$ are the same sequent when $\Gamma=\varphi_0\lor\varphi_1$). As we allow to add an arbitrary sequent~$\Delta$ in the conclusion of each rule, it is immediate that $\vdash\Gamma$ and $\Gamma\subseteq\Gamma'$ imply $\vdash\Gamma'$ (with a derivation of the same height). This property is known as~\emph{weakening} (as it yields a longer and hence weaker disjunction). In fact, it is sufficient to allow the addition of formulas $\Delta$ at axioms only. We allow to add $\Delta$ at any point in the proof, in order to avoid technicalities in connection with the variable condition (cf.~\cite[Section~2.3]{schwichtenberg77}). In any case, once weakening is available, it justifies the alternative forms
\begin{equation*}
\AxiomC{$\Gamma_0,\varphi_0$}
\AxiomC{$\Gamma_1,\varphi_1$}
\BinaryInfC{$\Delta,\Gamma_0,\Gamma_1,\varphi_0\land\varphi_1$}
\DisplayProof\qquad\text{and}\qquad
\AxiomC{$\Gamma_0,\varphi$}
\AxiomC{$\Gamma_1,\neg\varphi$}
\RightLabel{.}
\BinaryInfC{$\Delta,\Gamma_0,\Gamma_1$}
\DisplayProof
\end{equation*}
of conjunction and cut. In the rule for the universal quantifier, the substitution $[x/y]$ reflects the fact that bound variables may be renamed. The following terminology can be useful to describe derivations: In any of the rules above, the displayed formula in the lower sequent is called \emph{principal} ($\varphi_0\land\varphi_1$ for conjunction and none for cut). The displayed formulas in the upper sequent of each rule are called \emph{minor} ($\varphi$ and $\neg\varphi$ for cut). The remaining formulas (i.\,e., those in~$\Delta,\Gamma$) are called~\emph{side formulas}. Concrete derivations in sequent calculus are usually displayed as follows:

\begin{example}\label{ex:drinker}
The so-called drinker formula can be derived as
\begin{prooftree}
\AxiomC{$\neg Py,Py$}
\UnaryInfC{$\neg Py\lor\forall y.Py,Py$}
\UnaryInfC{$\exists_x(Px\to\forall y.Py),Py$}
\RightLabel{($y$ not free in $\exists_x(Px\to\forall y.Py)$)}
\UnaryInfC{$\exists_x(Px\to\forall y.Py),\forall y.Py$}
\UnaryInfC{$\exists_x(Px\to\forall y.Py),\neg Px\lor\forall y.Py$}
\UnaryInfC{$\exists_x(Px\to\forall y.Py)$}
\end{prooftree}
Note that $Px\to\forall y.Py$ and $\neg Px\lor\forall y.Py$ are literally the same formula, since we treat implication as meta-notation rather than a logical symbol.
\end{example}

Part~(c) of the following exercise provides a form of completeness. A direct and very elegant proof of completeness for sequent calculus will be given later.

\begin{exercise}\label{exer:sequent-basic}
(a) Show that the given sequent calculus for predicate logic is sound, i.\,e., that $\bigvee\Gamma$ is logically valid whenever we have $\vdash\Gamma$.

\noindent (b) Derive $\exists x(Px\to Pfx)$ and other logically valid formulas in sequent calculus.

\noindent (c) Consider natural deduction with logical symbols $\land,\to,\forall$ and $\bot$. Let $\varphi^\sim$ be the negation normal form of a formula $\varphi$ that is built from these symbols. More precisely, we keep $\bot$ as an uninterpreted predicate symbol, so that $(\varphi\to\bot)^\sim$ is $\neg\varphi^\sim\lor\bot$ (still with $\neg$ as meta-notation). Show how a natural deduction proof of $\varphi$ from open assumptions $\psi_1,\ldots,\psi_n$ can be transformed into a sequent calculus proof of $\bot,\neg\psi_1^\sim,\ldots,\neg\psi_n^\sim,\varphi^\sim$. \emph{Hint:} As preparation, show that $\vdash\varphi,\neg\varphi$ holds for any formula~$\varphi$ (the axioms cover prime formulas only). Deduce that $\vdash\Gamma,\varphi_0\lor\varphi_1$ implies $\vdash\Gamma,\varphi_0,\varphi_1$ (use cut even though this is not necessary, cf.~\cite[Lemma~1.3(c)]{buchholz-lecture-03}).
\end{exercise}

The exercise provides even more evidence that we are concerned with classical logic. Let us briefly indicate a suitable sequent calculus for the intuitionistic case:

\begin{remark}
For some purposes, it is preferable to consider sequents of the form $\varphi_1,\ldots,\varphi_m\Rightarrow\psi_1,\ldots,\psi_n$, to be interpreted as $\varphi_1\land\ldots\land\varphi_m\to\psi_1\lor\ldots\lor\psi_n$. The disadvantage is that this increases the number of rules, which makes case distinctions tedious. The great advantage is that intuitionistic logic is included in a very elegant way: to switch from the classical to the intuitionistic case, it suffices to demand $n\leq 1$, so that any sequent has the form $\varphi_1,\ldots,\varphi_m\Rightarrow\psi$ or $\varphi_1,\ldots,\varphi_m\Rightarrow\bot$ (for $n=0$). We refer to~\cite{schwichtenberg-troelstra} for details.
\end{remark}

Recall that our overall aim is to prove independence results. If we want to achieve this by using sequent calculus, the latter needs to be complete. To show that it is, we employ Kurt Sch\"utte's elegant method of ``deduction chains" (see~\cite{schuette56,schuette77}).

\begin{theorem}\label{thm:PL-complete}
The given sequent calculus for predicate logic is complete. Indeed, if $\bigvee\Gamma$ is logically valid, there is a derivation $\mathsf{PL}\vdash\Gamma$ that does not use the cut rule.
\end{theorem}

The role of cuts is further discussed below. Note that $\vdash\neg\psi_1,\ldots,\neg\psi_n,\varphi$ amounts to a derivation of $\varphi$ from open assumptions~$\psi_1,\ldots,\psi_n$. In this sense, the theorem covers derivations with axioms. In particular, we may include axioms for equality. If we do, the model below will interpret equality as an equivalence relation, and one can take the quotient as usual. The following construction can be generalized to yield models of infinite theories (now using the cut rule; cf.~\cite[Section~4]{rathjen-model-bi}). We avoid this generalization and appeal to compactness instead (the only drawback is that we do not get a new proof of the latter).

\begin{proof}
The idea is to search for a derivation of $\Gamma$ by applying rules backwards. If the search does not terminate, the attempted derivation tree has an infinite branch, which will determine a countermodel to~$\bigvee\Gamma$. To make this precise, let $2^{<\omega}$ be the tree of finite sequences with entries from~$\{0,1\}$. Given $\sigma=\langle\sigma_0,\ldots,\sigma_{n-1}\rangle\in 2^{<\omega}$ and $i\in\{0,1\}$, we put $\sigma\star i:=\langle \sigma_0,\ldots,\sigma_{n-1},i\rangle$. By recursion on sequences, we will define a tree $\mathcal D\subseteq 2^{<\omega}$ and a sequent $\Gamma(\sigma)$ for each $\sigma\in\mathcal D$. Here we view sequents as sequences (rather than sets) of formulas. To apply the previous definitions, it suffices to forget the order. Let us also fix an enumeration $t_0,t_1,\ldots$ of all terms (countability can be ensured by restricting to symbols from~$\Gamma$). In the base case of the recursion, we stipulate $\langle\rangle\in\mathcal D$ and $\Gamma(\langle\rangle):=\Gamma$ (the sequent from the theorem). For the recursion step, assume that $\sigma\in\mathcal D$ is given and $\Gamma(\sigma)$ is known. We declare that $\sigma\in\mathcal D$ is a leaf if $\Gamma(\sigma)$ contains some prime formula and its negation. Otherwise, write $\Gamma(\sigma)=\varphi,\Delta$ and distinguish the following cases (here $\varphi$ is the first entry of our `ordered sequent'; it is not hard to see that $\Gamma(\sigma)$ is never empty, except in the trivial case of $\Gamma=\emptyset$):
\begin{enumerate}[label=(\roman*)]
\item If $\varphi$ is a literal, declare $\sigma\star 0\in\mathcal D$ (and $\sigma\star 1\notin\mathcal D$) with $\Gamma(\sigma\star 0):=\Delta,\varphi$.
\item For $\varphi=\varphi_0\land\varphi_1$, let $\mathcal D$ contain $\sigma\star 0$ and $\sigma\star 1$ with $\Gamma(\sigma\star i):=\Delta,\varphi,\varphi_i$.
\item For $\varphi=\varphi_0\lor\varphi_1$, put $\sigma\star 0\in\mathcal D$ with
\begin{equation*}
\Gamma(\sigma\star 0):=\begin{cases}
\Delta,\varphi,\varphi_0 & \text{if $\Delta$ does not contain $\varphi_0$,}\\
\Delta,\varphi,\varphi_1 & \text{otherwise.}
\end{cases}
\end{equation*}
\item For $\varphi=\forall x.\psi$, put $\sigma\star 0\in\mathcal D$ and $\Gamma(\sigma\star 0):=\Delta,\varphi,\psi[x/y]$ for the first variable $y$ that is not free in $\varphi,\Delta$ (with respect to some fixed enumeration).
\item For $\varphi=\exists x.\psi$, put $\sigma\star 0\in\mathcal D$ and $\Gamma(\sigma\star 0):=\Delta,\varphi,\psi[x/t_k]$ for the smallest~$k$ such that $\Delta$ does not contain $\psi[x/t_k]$.
\end{enumerate}
If $\mathcal D$ is finite, we essentially have a derivation of $\Gamma$ in sequent calculus (except that the trivial `inferences' in (i) should be omitted). Note that this derivation does not involve cuts. Now assume $\mathcal D$ is infinite. Then K\H{o}nig's lemma (in its `weak' form for binary trees) yields an infinite branch, i.\,e., a function $f:\mathbb N\to\{0,1\}$ with $f[n]:=\langle f(0),\ldots,f(n-1)\rangle\in\mathcal D$ for all $n\in\mathbb N$. Put $\mathcal F:=\bigcup_{n\in\mathbb N}\Gamma(f[n])$ (viewing sequents as sets). To define a model $\mathcal M$ over the set of terms, let $f^{\mathcal M}\mathbf t:=f\mathbf t$ and
\begin{equation*}
\mathbf t\in P^{\mathcal M}\quad:\Leftrightarrow\quad P\mathbf t\notin\mathcal F.
\end{equation*}
I claim that $\varphi\in\mathcal F$ implies $\mathcal M\nvDash\varphi$, where all variables (and hence all terms) are interpreted as themselves. In view of $\Gamma=\Gamma(f[0])\subseteq\mathcal F$, this entails $\mathcal M\nvDash\bigvee\Gamma$, which completes the proof. The open claim is shown by induction over the height of formulas. We treat two crucial cases and leave the rest to the reader. First consider $\varphi=\neg P\mathbf t\in\mathcal F$. It suffices to show $P\mathbf t\notin\mathcal F$, since this entails $\mathcal M\vDash P\mathbf t$ and hence $\mathcal M\nvDash\varphi$. The crucial observation is that $\Gamma(f[n])\subseteq\Gamma(f[n+1])$ holds by construction (some authors avoid this property, which makes their proof search more efficient but harder to describe). So if $P\mathbf t$ and $\neg P\mathbf t$ did lie in $\mathcal F$, they would lie in a common sequent $\Gamma(f[n])$. But then $f[n]$ would be a leaf by construction, against the assumption that~$f$ is a branch. Finally, consider a formula~$\varphi=\exists x.\psi\in\mathcal F$. We get $\mathcal M\nvDash\varphi$ if $\mathcal M\nvDash\psi[x/t_k]$ holds for all $k\in\mathbb N$, which reduces to $\psi[x/t_k]\in\mathcal F$ by induction. Aiming at a contradiction, assume that $k$ is minimal with $\psi[x/t_k]\notin\mathcal F$. As before, some sequent $\Gamma(f[n])$ will contain $\varphi$ and $\psi[x/t_i]$ for all~$i<k$. Since we constantly rotate sequents, we may increase~$n$ to get $\Gamma(f[n])=\varphi,\Delta$. Now the construction yields $\psi[x/t_k]\in\Gamma(f[n+1])\subseteq\mathcal F$, as needed to conclude.
\end{proof}

Let us draw the following conclusion:

\begin{corollary}
If we have $\mathsf{PL}\vdash\Gamma$, then $\Gamma$ can be derived without cut.
\end{corollary}
\begin{proof}
By soundness, the assumption entails that $\bigvee\Gamma$ is logically valid. A cut free derivation can now be obtained by the previous theorem.
\end{proof}

We have given a so-called \emph{semantic} proof of \emph{cut elimination}. This proof does not use the given derivation $\mathsf{PL}\vdash\Gamma$ with cuts, except as a guarantee for validity. In contrast, a \emph{syntactic} proof is one that starts with the given derivation and removes cuts by a sequence of concrete transformations. We will later give such a proof in a different setting (see also the discussion of ``cuts as lemmata" below). For a syntactic proof in the setting of predicate logic we refer to~\cite{schwichtenberg77}. Syntactic proofs will usually yield more information, e.\,g., on the length of proofs. Also, semantic proofs are not always available: In the next section, we discuss a sequent calculus (with infinite derivations) that is sound and complete for the standard structure of natural numbers. Truth in this structure is not definable by a formula of first order arithmetic, as shown by Alfred Tarski~\cite{tarski36}. For this reason, Peano arithmetic cannot accommodate the usual proof of soundness, which is used in semantic cut elimination. At the same time, syntactic cut elimination for infinite derivations can be formalized in Peano arithmetic (up to issues of well foundedness; see e.\,g.~\cite{buchholz91}).

Theorem~\ref{thm:PL-complete} may give the impression that the cut rule is superfluous. This impression is misleading for various reasons: First, some applications require versions of sequent calculus in which not all cuts can be eliminated (cf.~the \emph{partial cut elimination} property in Exercise~\ref{ex:ISigma_1}). Secondly, even when all cuts can be eliminated in the end, the cut rule may play an important role in intermediate steps. As an example, the embedding of natural deduction into sequent calculus (part~(c) of Exercise~\ref{exer:sequent-basic}) is hardly feasible without the cut rule. This is related to our final point: the shortest cut free proof can be super-exponentially longer than a proof with cuts (due to Richard Statman and Vladimir Orevkov, see e.\,g.~\cite[Section~5]{pudlak-tl-98}; Exercise~\ref{exer:herbrand}(b) reveals a connection with the size of Herbrand disjunctions).

Informally, cuts can be interpreted as applications of lemmata in the proof of a theorem. This is best explained in an example: consider the (incomplete) derivation
\begin{prooftree}
\AxiomC{$\neg\psi[x/t],\varphi$}
\UnaryInfC{$\exists x.\neg\psi,\varphi$}
\AxiomC{$\psi$}
\UnaryInfC{$\forall x.\psi$}
\RightLabel{(Cut).}
\BinaryInfC{$\varphi$}
\end{prooftree}
The idea is that $\forall x.\psi$ is a lemma, which we establish by proving~$\psi$ for arbitrary~$x$ (as captured by the variable condition). To deduce our theorem~$\varphi$, we only use a specific instance~$\psi[x/t]$ of this lemma (note that $\neg\psi[x/t],\varphi$ corresponds to $\psi[x/t]\to\varphi$). Now we could substitute $x:=t$ throughout the proof of the lemma, to get a proof of~$\psi[x/t]$. Our theorem $\varphi$ could then be deduced as
\begin{prooftree}
\AxiomC{$\neg\psi[x/t],\varphi$}
\AxiomC{$\psi[x/t]$}
\RightLabel{(Cut).}
\BinaryInfC{$\varphi$}
\end{prooftree}
Note that the new cut formula $\psi[x/t]$ has smaller height than the cut formula $\forall x.\psi$ above. Syntactic cut elimination removes all cuts or ``lemmata" by iterating (more general versions of) the given transformation. In some sense, this makes the proof more direct. However, it can also make it less transparent and less modular (general lemmata can be reused). Also recall that cut free proofs can be unfeasible due to their extreme length. This can be seen as evidence that lemmata are indispensable for mathematical practice. Particularly interesting situations arise when lemmata involve ``abstract concepts" while the theorem is very ``concrete".

At the same time, certain information can be tracked through cut free derivations, while cuts interrupt the ``flow of information". More concretely, certain proofs by induction over derivations will only go through if there are no cuts. The following important cases are given as exercises, because the results will not be used in this course. For part~(b) of the exercise, it may help to adopt the following view on rules: if the last rule in a derivation of~$\Gamma$ introduces the formula $\varphi_0\lor\varphi_1$ (which is thus contained in~$\Gamma$), then this rule has premise $\Delta\subseteq\Gamma,\varphi_i$ for some $i\in\{0,1\}$.

\begin{exercise}\label{exer:herbrand}
(a) Show that any formula in a cut free derivation of~$\varphi$ is a subformula of~$\varphi$ (where any instance $\psi[x/t]$ counts as a subformula of~$\forall x.\psi$ and $\exists x.\psi$).

(b) Prove the following version of Herbrand's theorem (more existential quantifiers could be accommodated): For a formula~$\theta$ without quantifiers, a cut free derivation of~$\exists x.\theta$ can be transformed into a derivation of $\theta[x/t_1],\ldots,\theta[x/t_n]$ for suitable terms~$t_i$. \emph{Hint:} Use induction to show that a derivation of $\Gamma$ yields a derivation of $\Gamma\backslash\{\exists x.\theta\},\theta[x/t_1],\ldots,\theta[x/t_n]$, if $\Gamma$ contains no quantifiers except in~$\exists x.\theta$.
\end{exercise}

\section{Induction and infinite derivations}\label{sect:inf-derivs}

The present section introduces Peano arithmetic and an important tool to investigate it: a system of infinite derivations. In particular, we will show that induction along `large' well orders cannot be proved by cut free derivations of `small' infinite height. This is one central ingredient for our proof that Peano arithmetic cannot establish Kruskal's theorem for binary trees.

Our approach is rather robust with respect to the chosen presentation of Peano arithmetic. For definiteness, we agree to follow the textbook by Petr H\'ajek and Pavel Pudl\'ak~\cite{hajek91}: Let~$\mathcal L_{\mathsf{PA}}$ be the first order language that consists of a constant~$0$, function symbols~$S$ (unary successor) and $+,\times$ (binary addition and multiplication), as well as binary relation symbols~$\leq$ and~$=$. We write $\overline n$ for the $n$-th numeral, which is given by the recursive clauses $\overline 0=0$ and $\overline{n+1}=S\overline n$.

As in the previous section, formulas are built up from literals by the connectives $\land,\lor$ and the quantifiers $\forall,\exists$; other connectives may be used as abbreviations. When we say that a formula is true, we refer to the standard interpretation of $\mathcal L_{\sf{PA}}$ in the structure of natural numbers. The axioms of Peano arithmetic ($\mathsf{PA}$) consist of the usual equality axioms, the axioms of Robinson arithmetic (as given in~\cite[Definition~I.1.1]{hajek91}), and all instances
\begin{equation}\tag{$\mathcal I\varphi$}
\varphi[x/0]\land\forall x(\varphi\to\varphi[x/Sx])\to\forall x.\varphi
\end{equation}
of induction, where $\varphi$ can be any $\mathcal L_{\mathsf{PA}}$-formula (possibly with further free variables).

Note that the induction scheme does only apply to properties that are defined by a formula. For our approach, it will be crucial to transcend this limitation: Let $\mathcal L_{\mathsf{PA}}^X$ be the extension of $\mathcal L_{\mathsf{PA}}$ by a unary relation symbol $X$, which does not have a fixed interpretation over~$\mathbb N$. The $\mathcal L_{\mathsf{PA}}^X$-theory $\mathsf{PA}[X]$ consists of the equality axioms for the extended language, the axioms of Robinson arithmetic, and the induction axioms $\mathcal I\varphi$ for all $\mathcal L_{\mathsf{PA}}^X$-formulas~$\varphi$.

To motivate our analysis of Peano arithmetic, we consider a somewhat simpler case first. The following is a classical and important result. We present it in the form of an exercise, because it will not be used in the rest of this lecture (for complete presentations, see e.\,g.~\cite[Paragraph~2]{buchholz-lecture-03} or~\cite[Section~3.1]{buss-arithm-98}).

\begin{exercise}\label{ex:ISigma_1}
An occurrence of a quantifier in an $\mathcal L_{\mathsf{PA}}$-formula is called bounded if it has the form
\begin{equation*}
\forall x\leq t.\varphi\,\equiv\,\forall x(x\leq t\to\varphi)\quad\text{or}\quad\exists x\leq t.\varphi\,\equiv\,\exists x(x\leq t\land\varphi).
\end{equation*}
A formula is $\Delta_0$ if all quantifiers are bounded. It is $\Sigma_1$ (or $\Pi_1$) if it has the form $\exists x.\theta$ (resp.~$\forall x.\theta$) for a $\Delta_0$-formula~$\theta$. The theory $\mathsf{I\Sigma_1}$ is defined as $\mathsf{PA}$, except that $\mathcal I\varphi$ is only an axiom when $\varphi$ is $\Sigma_1$. We point out that $\mathsf{I\Sigma_1}$ coincides with $\mathsf{I\Pi_1}$ (defined analogously; cf.~\cite[Lemma~I.2.12]{hajek91}). The aim of this exercise is to prove the following result of Charles Parsons~\cite{parsons70}: If $\mathsf{I\Sigma_1}$ proves $\forall x\exists y.\varphi$ for a $\Sigma_1$-formula~$\varphi$, there is a primitive recursive function $f:\mathbb N\to\mathbb N$ such that $\exists y\leq\overline{f(n)}.\varphi[x/\overline n]$ holds for all~$n\in\mathbb N$. In order to achieve this, we extend the sequent calculus from Definition~\ref{def:seq-calc} as follows: First, all equality axioms and all axioms of Robinson arithmetic are added as axioms (rules without premises). We assume that these axioms contain no unbounded existential quantifiers (note that we may replace $\exists z$ by $\exists z\leq y$ in axiom (Q8) of~\cite[Definition~I.1.1]{hajek91}). Secondly, we extend our sequent calculus by the induction rule
\begin{equation*}
\AxiomC{$\Gamma,\varphi[x/0]$}
\AxiomC{$\Gamma,\neg\varphi,\varphi[x/Sx]$}
\RightLabel{($x$ not free in~$\Gamma,\forall x.\varphi$),}
\BinaryInfC{$\Delta,\Gamma,\varphi[x/t]$}
\DisplayProof
\end{equation*}
where $\varphi$ must be $\Sigma_1$ and $t$ can be any term. The following is worth checking:
\begin{enumerate}[label=(\alph*)]
\item Any logical consequence of~$\mathsf{I\Sigma_1}$ can be derived in our extended sequent calculus. \emph{Hint:} Derive the axioms, in particular the induction axioms in the usual form. Then invoke Theorem~\ref{thm:PL-complete} and employ cuts.
\end{enumerate}
In the context of this exercise, we say that a cut
\begin{equation*}
\AxiomC{$\Gamma,\varphi$}
\AxiomC{$\Gamma,\neg\varphi$}
\BinaryInfC{$\Delta,\Gamma$}
\DisplayProof
\end{equation*}
is free if~$\varphi$ (or equivalently~$\neg\varphi$) contains both unbounded universal and unbounded existential quantifiers. Crucially, the latter holds for no principal formula of an axiom or rule of our extended sequent calculus (thanks to the formulation of induction as a rule). This ensures the following property (`partial cut elimination', see~\cite[Theorem~1.8]{buchholz-lecture-03}), which you may use without proof: If $\Gamma$ can be derived in our extended sequent calculus, then it has a derivation without free cuts. In order to derive Parson's result, we need one last ingredient (`asymmetric interpretation'): Given $m,n\in\mathbb N$, write $\varphi^{m,n}$ for the formula that results from~$\varphi$ when we replace all unbounded quantifiers $\forall x$ and $\exists y$ by the bounded quantifiers $\forall x\leq\overline m$ and $\exists y\leq\overline n$, respectively. We say that a sentence~$\varphi$ is witnessed by $f:\mathbb N\to\mathbb N$ if $\varphi^{m,f(m)}$ is true for all~$m\in\mathbb N$. A formula with free variables is witnessed by~$f$ if the same holds for its universal closure. Show the following:
\begin{enumerate}[label=(\alph*)]\setcounter{enumi}{1}
\item If $\Gamma$ is derivable in our extended sequent calculus, then $\bigvee\Gamma$ is witnessed by some primitive recursive function. Hence the aforementioned result by Parsons holds. \emph{Hint:} Argue by induction over derivations without free cuts. You will use that $\varphi^{m,n}$ implies $\varphi^{k,l}$ when $k\leq m$ and $n\leq l$. For this it is crucial that formulas are in negation normal form, so that all subformulas occur positively. Apart from the induction and cut rules, one should pay special attention to the introduction of a \emph{bounded} universal quantifier.
\end{enumerate}
To prove the induction step for a cut over~$\varphi$ and $\neg\varphi$, you will have used that
\begin{equation*}
(\neg\varphi)^{f(m),g(f(m))}\,\equiv\,\neg\left(\varphi^{m,f(m)}\right)
\end{equation*}
holds when $\varphi$ contains no unbounded universal quantifiers. The given equality does not hold for general~$\varphi$. This explains why the approach does not extend from~$\mathsf{I\Sigma_1}$ to full Peano arithmetic: there the induction formulas can have arbitrary quantifier complexity, so that partial cut elimination has no real effect.
\end{exercise}

The exercise provides some evidence that full Peano arithmetic is beyond the reach of cut elimination as we have seen it so far. More precisely, the axioms of $\pa$ have arbitrary quantifier complexity, which blocks even partial cut elimination. A somewhat radical but very elegant solution is to prove the axioms in a system of infinite derivations. Specifically, we will consider derivations with the $\omega$-rule
\begin{equation*}
\AxiomC{$\Gamma,\varphi[x/\overline 0]\qquad\Gamma,\varphi[x/\overline 1]\qquad\cdots\qquad\Gamma,\varphi[x/\overline n]\qquad\cdots\qquad\text{(all $n\in\mathbb N$)}$}
\RightLabel{,}
\UnaryInfC{$\Delta,\Gamma,\forall x.\varphi$}
\DisplayProof
\end{equation*}
which allows to conclude $\forall x.\varphi$ when one has a derivation of $\varphi[x/\overline n]$ for each individual number~$n\in\mathbb N$. The $\omega$-rule goes back to David Hilbert~\cite{hilbert-zahlenlehre} (similar ideas of L.\,E.\,J.~Brouwer and Ernst Zermelo are mentioned in~\cite[Footnote~15]{rathjen-sieg-stanford}). Its systematic use in ordinal analysis is due to Kurt Sch\"utte~\cite{schuette-omega-rule}.  Of course, infinite derivations cannot be written down in practice. At the same time, they are perfectly decent mathematical objects, with a precise definition given below.

To define an infinite proof system that is suitable for our purpose, we need some preparation: First, we fix a measure for the complexity of formulas, which will be used to control cut inferences:

\begin{definition}
The rank $\rk(\varphi)\in\mathbb N$ of an $\mathcal L_{\mathsf{PA}}^X$-formula $\varphi$ is recursively given by
\begin{gather*}
\rk(\varphi)=0\quad\text{for any literal~$\varphi$},\\
\rk(\varphi_0\land\varphi_1)=\rk(\varphi_0\lor\varphi_1)=\max\{\rk(\varphi_0),\rk(\varphi_1)\}+1,\\
\rk(\forall x.\varphi)=\rk(\exists x.\varphi)=\rk(\varphi)+1.
\end{gather*}
\end{definition}

A straightforward induction over formulas yields $\rk(\neg\varphi)=\rk(\varphi)=\rk(\varphi[x/t])$. Our derivations will be infinite but should still be well founded. As a second preparation, we fix a well order $(E,\prec)$ that will help to ensure this in a controlled way. The following assumption lists all properties that are needed in the present section; further requirements will be added later. We will eventually discharge all assumptions, by showing that they are satisfied for the order \mbox{$E=\varepsilon_0=\min\{\alpha\,|\,\omega^\alpha=\alpha\}$} (see Section~\ref{sect:eps_0} for explanations). If one wants to satisfy the following assumptions only, one can take the simpler order $E=\omega\cdot 2$ (one copy of~$\mathbb N\cong\omega$ above another).

\begin{assumption}\label{ass:E}
Throughout the following, we assume that $(E,\prec)$ is a well order with the following properties: First, there is a map $E\ni\alpha\mapsto\alpha+1\in E$ with $\alpha\prec\alpha+1$ for all~$\alpha\in E$. Secondly, there are elements $0,\omega\in E$ with $0\prec\omega$ such that $\alpha\prec\omega$ entails $\alpha+1\prec\omega$.
\end{assumption}

We will later add the assumption that~$0$ is the minimal element of~$E$. For the time begin, the only role of~$0$ is to ensure that~$\omega$ is not the minimal element. As a final preparation, we fix an $\mathcal L_{\pa}$-formula $x\tl y$ (with no further free variables) that defines a well order on~$\mathbb N$. In the following we assume that $\mathcal L_{\pa}^X$-formulas are closed, except when noted otherwise. This makes sense in view of the $\omega$-rule, which introduces a universal quantifier without the use of free variables. Note that the predicate $X$ may occur in closed formulas. We now present our proof system:

\begin{definition}\label{def:inf-proofs}
By recursion on~$\alpha\in E$, we declare that $\sststile{d}{\alpha}\Gamma$ holds precisely in the following cases (for $d\in\mathbb N$ and a sequent~$\Gamma$ that consists of $\mathcal L_{\pa}^X$-sentences):
\begin{enumerate}[label=(\roman*)]
\item $\Gamma$ contains a true literal from~$\mathcal L_{\pa}$ or two formulas $Xs$ and $\neg Xt$ for terms $s$ and $t$ with the same value.
\item $\Gamma$ contains a formula $\varphi_0\land\varphi_1$ ($\varphi_0\lor\varphi_1$) and we have $\sststile{d(i)}{\alpha(i)}\Delta_i$ with $\alpha(i)\prec\alpha$, $d(i)\leq d$ and $\Delta_i\subseteq\Gamma,\varphi_i$, for every (some) $i\in\{0,1\}$.
\item $\Gamma$ contains a formula $\forall x.\varphi$ ($\exists x.\varphi$) and we have $\sststile{d(t)}{\alpha(t)}\Delta_t$ with $\alpha(t)\prec\alpha$, $d(t)\leq d$ and $\Delta_t\subseteq\Gamma,\varphi[x/t]$, for every (some) closed term $t$.
\item $\Gamma$ contains a formula $Xt$ and we have $\sststile{d(s)}{\alpha(s)}\Delta_s$ with $\alpha(s)\prec\alpha,d(s)\leq d$ and $\Delta_s\subseteq\Gamma,Xs$, for every closed term $s$ with $s\tl t$.
\item For some $\mathcal L_{\pa}^X$-sentence $\varphi$ with $\rk(\varphi)<d$ we have $\sststile{d(i)}{\alpha(i)}\Delta_i$ for $i\in\{0,1\}$, with $\alpha(i)\prec\alpha$ and $d(i)\leq d$ as well as $\Delta_0\subseteq\Gamma,\varphi$ and $\Delta_1\subseteq\Gamma,\neg\varphi$.
\end{enumerate}
\end{definition}

Given $\alpha\prec\beta,d\leq e$ and $\Gamma\subseteq\Gamma'$, it is immediate that $\sststile{d}{\alpha}\Gamma$ entails $\sststile{e}{\beta}\Gamma'$. As before, we will refer to this fact as~\emph{weakening}. We note that weakening could also be ensured in a more `economic' way: For example, if~$d$ can be chosen arbitrary in~(i), it is not strictly necessary to increase it in the remaining clauses. In contrast, it is important that $\alpha(s)$ may depend on~$s$ in~(iv), as infinitely many $\alpha(s)\prec\alpha$ need not admit a common bound $\alpha'$ with $\alpha(s)\preceq\alpha'\prec\alpha$.

Clauses~(i-iii) and~(v) of Definition~\ref{def:inf-proofs} correspond to axioms, inferences for connectives and quantifiers, and the cut rule. Clause~(iv) amounts to the~\emph{progression rule} of Sch\"utte, which implements induction along~$\tl$ in a hard-wired way. At least intuitively, $\sststile{d}{\alpha}\Gamma$ holds if $\Gamma$ has an infinite derivation tree with height (sometimes called rank) bounded by~$\alpha$, where all cut formulas have rank below~$d$. In particular, $d=0$ means that there are no cuts.

Both~(iv) and the universal case of~(iii) are versions of the $\omega$-rule. Note that~(iii) has premises $\varphi[x/t]$ for all terms~$t$, not just for numerals~$t=\overline n$. We will show that a version for numerals can be deduced, based on part~(a) of the following exercise.

\begin{exercise}\label{exer:inf-basic}
(a) Show that $\sststile{d}{\alpha}\Gamma$ entails $\sststile{d}{\alpha}\Gamma\backslash\{\varphi[x/s]\},\varphi[x/t]$ whenever $s$ and $t$ are closed terms with the same value.

(b) For a true $\mathcal L_{\pa}$-sentence~$\varphi$, show that we have $\sststile{0}{\rk(\varphi)}\varphi$. Also show that we have $\sststile{0}{2\cdot\rk(\varphi)}\varphi,\neg\varphi$ for any $\mathcal L_{\pa}^X$-sentence~$\varphi$ (cf.~the hint in part~(c) of Exercise~\ref{exer:sequent-basic}). Here $n\in E$ is explained by $\alpha+0:=\alpha$, $\alpha+(n+1):=(\alpha+n)+1$ and $n:=0+n$.
\end{exercise}

In some sense, part~(b) of the exercise shows that the unprovability of (true) $\mathcal L_{\mathsf{PA}}$-sentences in Peano arithmetic cannot be established via the given system of infinite proofs. At the same time, we will use this system to show that a certain~$\mathcal L_{\mathsf{PA}}^X$-sentence is unprovable in~$\mathsf{PA}[X]$. The point is that the additional predicate symbol~$X$ does not have a fixed interpretation, which makes it possible to represent universal quantification over subsets of~$\mathbb N$. However, this observation should be taken with care: it is possible to obtain meaningful results on~$\mathcal L_{\textsf{PA}}$-sentences as well, either by formalizing our infinite proof system in~$\mathsf{PA}$ itself (as in~\cite{buchholz91}), or via a modified proof system that controls numerical information explicitly (as in~\cite{wainer-fairtlough-98}). As promised above, we now deduce a more convenient version of the $\omega$-rule:

\begin{lemma}\label{lem:omega-rule}
Assuming $\forall x.\varphi\in\Gamma$, we can conclude $\sststile{d}{\alpha}\Gamma$ if we have $\sststile{d(n)}{\alpha(n)}\Delta_n$ with $\alpha(n)\prec\alpha,d(n)\leq d$ and $\Delta_n\subseteq\Gamma,\varphi[x/\overline{n}]$ for all~$n\in\mathbb N$.
\end{lemma}
\begin{proof}
To conclude by clause~(iii) of Definition~\ref{def:inf-proofs}, we need to find suitable~$\alpha(t),d(t)$ and $\Delta_t$ for any given term~$t$. The latter has some numerical value~$n\in\mathbb N$. We put $\alpha(t):=\alpha(n), d(t):=d(n)$ and $\Delta_t:=\Delta_n\backslash\{\varphi[x/\overline n]\},\varphi[x/t]\subseteq\Gamma,\varphi[x/t]$. It just remains to observe that $\sststile{d(t)}{\alpha(t)}\Delta_t$ holds by part~(a) of the previous exercise.
\end{proof}

The following is needed to establish independence via infinite proofs:

\begin{theorem}[`Embedding']\label{thm:embedding}
If an $\mathcal L_{\mathsf{PA}}^X$-sentence~$\varphi$ is a theorem of~$\mathsf{PA}[X]$, then we have $\sststile{d}{\alpha}\varphi$ for some $\alpha\in E$ and $d\in\mathbb N$.
\end{theorem}
\begin{proof}
Assume that $\varphi$ follows from the $\mathsf{PA}[X]$-axioms $\psi_1,\ldots,\psi_n$. By Theorem~\ref{thm:PL-complete} (completeness), we get a cut-free derivation $\vdash\neg\psi_1,\ldots,\neg\psi_n,\varphi$ in the finitary sequent calculus from Definition~\ref{def:seq-calc}. We will show $\sststile{0}{\alpha(i)}\psi_i$ and $\sststile{0}{\alpha(0)}\neg\psi_1,\ldots,\neg\psi_n,\varphi$ for suitable $\alpha(i)$. Once this is achieved, we get $\sststile{d}{\alpha+n}\varphi$ by $n$~cuts (clause~(v) from Definition~\ref{def:inf-proofs}), for $\alpha:=\max_{\prec}\{\alpha(i)\,|\,i\leq n\}$ and $d:=\max\{\rk(\psi_i)+1\,|\,1\leq i\leq n\}$. Let us first provide derivations for the axioms: Those that do not involve~$X$ are covered by part~(b) of the exercise above. Concerning the equality axiom for~$X$, note that Definition~\ref{def:inf-proofs}(i) yields $\sststile{0}{0}\neg s=t,\neg Xs,Xt$ for any terms~$s,t$ (either $\neg s=t$ is a true literal or~$s$ and $t$ have the same value). We can now derive
\begin{equation*}
\AxiomC{$\sststile{0}{0}\neg s=t,\neg Xs,Xt$}
\UnaryInfC{$\sststile{0}{1}\neg s=t\lor\neg Xs,\neg Xs,Xt$}
\UnaryInfC{$\sststile{0}{2}\neg s=t\lor\neg Xs,Xt$}
\UnaryInfC{$\qquad\quad\cdots\quad\sststile{0}{4}(\neg s=t\lor\neg Xs)\lor Xt\quad\cdots\quad\text{(all~$s,t$)}$}
\RightLabel{.}
\UnaryInfC{$\sststile{0}{6}\forall x\forall y(x=y\land Xx\to Xy)$}
\DisplayProof
\end{equation*}
For the last step, recall that we treat implication as meta-notation (cf.~Example~\ref{ex:drinker}). To establish induction, recall that Exercise~\ref{exer:inf-basic}(b) yields $\sststile{0}{2\cdot\rk(\psi)}\psi[x/\overline n],\neg\psi[x/\overline n]$ for all $n\in N$. By induction (in the meta theory) we get derivations
\begin{equation*}
\AxiomC{$\vdots$}
\noLine
\UnaryInfC{$\sststile{0}{2\cdot(\rk(\psi)+n)}\neg\psi[x/0],\exists x(\psi\land\neg\psi[x/Sx]),\psi[x/\overline n]$}
\AxiomC{$\sststile{0}{2\cdot\rk(\psi)}\neg\psi[x/S\overline n],\psi[x/\overline{n+1}]$}
\BinaryInfC{$\sststile{0}{2\cdot(\rk(\psi)+n)+1}\neg\psi[x/0],\exists x(\psi\land\neg\psi[x/Sx]),\psi[x/\overline n]\land\neg\psi[x/S\overline n],\psi[x/\overline{n+1}]$}
\RightLabel{.}
\UnaryInfC{$\sststile{0}{2\cdot(\rk(\psi)+n+1)}\neg\psi[x/0],\exists x(\psi\land\neg\psi[x/Sx]),\psi[x/\overline{n+1}]$}
\DisplayProof
\end{equation*}
Using Lemma~\ref{lem:omega-rule} ($\omega$-rule), we can derive the induction axiom as
\begin{equation*}
\AxiomC{$\cdots\qquad\sststile{0}{2\cdot(\rk(\psi)+n)}\neg\psi[x/0],\exists x(\psi\land\neg\psi[x/Sx]),\psi[x/\overline{n}]\qquad\cdots$}
\UnaryInfC{$\sststile{0}{\omega}\neg\psi[x/0],\exists x(\psi\land\neg\psi[x/Sx]),\forall x.\psi$}
\RightLabel{.}
\UnaryInfC{$\sststile{0}{\omega+4}\psi[x/0]\land\forall x(\psi\to\psi[x/Sx])\to\forall x.\psi$}
\DisplayProof
\end{equation*}
It remains to translate finite derivations into infinite ones. Consider a sequent~$\Gamma$ of $\mathcal L_{\mathsf{PA}}^X$-formulas, possibly with free variables. Any sequent that arises from $\Gamma$ when we substitute closed terms for all free variables is called a closed instance of~$\Gamma$. By induction over a derivation $\vdash\Gamma$ (in the sense of Definition~\ref{def:seq-calc}) with cut formulas of rank below~$d$, we find an $\alpha\prec\omega$ such that $\sststile{d}{\alpha}\Gamma'$ (in the sense of Definition~\ref{def:inf-proofs}) holds for any closed instance~$\Gamma'$ of~$\Gamma$. We discuss one case and leave all others to the reader: Assume that $\Delta,\Gamma,\forall x.\varphi$ was deduced from $\Gamma,\varphi[x/y]$, where $y$ is not among the variables~$x,\mathbf z$ that are free in $\Gamma,\forall x.\varphi$. The induction hypothesis yields an~$\alpha\prec\omega$ with $\sststile{d}{\alpha}\Gamma[x/r,\mathbf z/\mathbf t],\varphi[x/y][y/s,\mathbf z/\mathbf t]$ for all closed terms~$r,s$ and $\mathbf t$. Due to
\begin{equation*}
(\forall x.\varphi)[x/r,\mathbf z/\mathbf t]\,\equiv\,\forall x(\varphi[\mathbf z/\mathbf t])\quad\text{and}\quad\varphi[\mathbf z/\mathbf t][x/s]\,\equiv\,\varphi[x/y][y/s,\mathbf z/\mathbf t],
\end{equation*}
the $\omega$-rule yields $\sststile{d}{\alpha+1}\Delta',(\Gamma,\forall x.\varphi)[x/r,\mathbf z/\mathbf t]$, for any closed instance~$\Delta'$ of~$\Delta$.
\end{proof}

Our overall goal is to show that Peano arithmetic cannot prove Kruskal's theorem for binary trees. As we will see, the latter is closely related to `transfinite' induction along a certain well order. In the rest of this section, we show that transfinite induction cannot be proved by cut free derivations of `small' infinite height. To express transfinite induction, we abbreviate $\forall y\tl x.\varphi\equiv\forall y(y\tl x\to\varphi)$ and
\begin{align*}
\prog_{\tl}\,&:\equiv\,\forall x(\forall y\tl x.Xy\to Xx),\\
\ti_{\tl}\,&:\equiv\,\prog_{\tl}\to\forall x.Xx.
\end{align*}
The reader may wish to recover the usual connection between transfinite induction and well foundedness in our setting:

\begin{exercise}\label{ex:ti-wf}
Let $\mathcal N$ be the standard model of~$\mathcal L_{\pa}$ (in which all symbols receive their usual interpretation over~$\mathbb N$). By a standard model of $\mathcal L_{\pa}^X$ we mean any extension of~$\mathcal N$ to that language (i.\,e., the predicate symbol~$X$ can be interpreted by any set of natural numbers). According to a standing assumption, $x\tl y$ is an $\mathcal L_{\pa}$-formula that defines a well order~$\tl$ on~$\mathbb N$. For the present exercise, we drop the assumption that~$\tl$ is well founded (but it should still be a linear order). Show that the following are equivalent:
\begin{enumerate}[label=(\roman*)]
\item The sentence $\ti_{\tl}$ holds in all standard models of $\mathcal L_{\pa}^X$.
\item Any non-empty subset of $\mathbb N$ has a $\tl$-minimal element.
\end{enumerate}
Also recall that (ii) is equivalent to the other common formulation of well foundedness: there is no infinite sequence $n_0,n_1,\ldots\subseteq\mathbb N$ with $n_{i+1}\tl n_i$ for all~$i\in\mathbb N$.
\end{exercise}

In our infinitary proof system we can show that $\tl$ is \emph{progressive}:

\begin{lemma}\label{lem:prove-prog}
We have $\sststile{0}{\alpha}\prog_{\tl}$ for some $\alpha\prec\omega$.
\end{lemma}
\begin{proof}
For $\alpha=\rk(x\tl y)$ we have $\sststile{0}{\alpha}s\tl t$ whenever $s\tl t$ is true, by part~(b) of Exercise~\ref{exer:inf-basic}. Thus the desired derivation can be given as
\begin{equation*}
\AxiomC{$\sststile{0}{\alpha}s\tl t$}
\AxiomC{$\sststile{0}{0}\neg Xs,Xs$}
\BinaryInfC{$\sststile{0}{\alpha+1}s\tl t\land\neg Xs,Xs$}
\UnaryInfC{$\qquad\qquad\quad\cdots\quad\sststile{0}{\alpha+2}\exists y(y\tl t\land\neg Xy),Xs\quad\cdots\quad\text{(all $s\tl t$)}$}
\UnaryInfC{$\sststile{0}{\alpha+3}\exists y(y\tl t\land\neg Xy),Xt$}
\UnaryInfC{$\qquad\quad\cdots\quad\sststile{0}{\alpha+5}\exists y(y\tl t\land\neg Xy)\lor Xt\quad\cdots\quad\text{(all $t$)}$}
\RightLabel{.}
\UnaryInfC{$\sststile{0}{\alpha+6}\prog_{\tl}$}
\DisplayProof
\end{equation*}
Note that the third and last inference are justified by the progression and~$\omega$-rule (clauses (iv) and~(iii) of Definition~\ref{def:inf-proofs}).
\end{proof}

To connect transfinite induction and the height of cut free proofs, we describe an `attempted' embedding of $(\mathbb N,\tl)$ into $(E,\prec)$: If $o(m)\in E$ is defined for all $m\tl n$, we recursively set
\begin{equation*}
o(n):=\begin{cases}
\min_\prec\{\alpha\in E\,:\,o(m)\prec\alpha\text{ for all }m\tl n\} & \text{if such an $\alpha$ exists},\\
\text{undefined} & \text{otherwise}.
\end{cases}
\end{equation*}
If $o(m)$ is undefined for some $m\tl n$, then $o(n)$ is undefined as well. Note that the recursion is justified since~$\tl$ is a well order; the minimum above exists since the same holds for~$\prec$. The following is straightforward but crucial (in writing $o(t)$ we identify the term~$t$ with its value):

\begin{lemma}\label{lem:ti-cut-free}
If we have $\sststile{0}{\alpha}Xt$, then $o(t)\preceq\alpha$ is defined.
\end{lemma}
\begin{proof}
Let us establish a somewhat more general claim by induction on~$\alpha$: if we have \mbox{$\sststile{0}{\alpha}Xt_0,\ldots,Xt_n$}, then $o(t_i)\preceq\alpha$ is defined for some~$i\leq n$. Without cuts, the given sequent can only be deduced by the progression rule: for some $i\leq n$ and all $s\tl t_i$ we must have $\sststile{0}{\alpha(s)}\Delta_s$ with $\alpha(s)\prec\alpha$ and $\Delta_s\subseteq Xt_0,\ldots,Xt_n,Xs$. Unless we already have~$o(t_j)\preceq\alpha(s)\prec\alpha$ for some $j\leq n$, the induction hypothesis tells us that $o(s)\prec\alpha$ is defined for all~$s\tl t_i$. But then $o(t_i)\preceq\alpha$ is defined.
\end{proof}

Let us abbreviate $E\!\restriction\!\alpha:=\{\gamma\in E\,|\,\gamma\prec\alpha\}$ for~$\alpha\in E$. The restriction of~$\prec$ to $E\!\restriction\!\alpha$ will also be denoted by~$\prec$. An embedding between well orders $(X,<_X)$ and $(Y,<_Y)$ is defined as a function~$f:X\to Y$ such that $x<_X y$ entails $f(x)<_Y f(y)$. We write $(X,<_X)\lesssim(Y,<_Y)$ if such an embedding exists (note that the notation $\lesssim$ is not standard). Part~(a) of the following exercise is needed to justify the last sentence of Theorem~\ref{thm:bound-TI} below. The other parts provide additional background.

\begin{exercise}
Prove the following:

(a) We have $(E,\prec)\not\lesssim(E\!\restriction\!\alpha,\prec)$ for any~$\alpha\in E$. \emph{Hint:} If~$f:E\to E\!\restriction\!\alpha$ was an embedding, we would get $\gamma\preceq f(\gamma)$ by induction on~$\gamma\in E$.

(b) For any embedding $f:X\to Y$ between well orders (or even linear orders), it is also the case that $f(x)<_Y f(x')$ entails $x<_X x'$.

(c) We have $(X,<_X)\lesssim(Y,<_Y)$ if and only if we have either $(X,<_X)\cong(Y,<_Y)$ or $(X,<_X)\cong(Y\!\restriction\!y,<_Y)$ for some~$y\in Y$.

(d) The relation $\lesssim$ induces a well order on isomorphism classes of well orders.
\end{exercise}

We conclude this section with a conditional independence result. The condition that cut elimination holds will be proved in the next section, under additional assumptions on the well order~$(E,\prec)$.

\begin{theorem}\label{thm:bound-TI}
Assume that our infinite proof system admits cut elimination, in the sense that $\sststile{d}{\alpha}\Gamma$ entails $\sststile{0}{f(\alpha,d)}\Gamma$, where $f(\alpha,d)\in E$ does not depend on~$\Gamma$. If $\ti_{\tl}$ is a theorem of $\mathsf{PA}[X]$, then we have $(\mathbb N,\tl)\lesssim(E\!\restriction\!\alpha,\prec)$ for some~$\alpha\in E$. Hence $\mathsf{PA}[X]$ cannot prove $\ti_{\tl}$ when we have $(E,\prec)\lesssim(\mathbb N,\tl)$.
\end{theorem}
\begin{proof}
Assume that $\ti_{\tl}$ is provable in~$\mathsf{PA}[X]$. By Theorem~\ref{thm:embedding} (embedding) we obtain $\sststile{d}{\alpha}\ti_{\tl}$ for some~$\alpha\in E$ and $d\in\mathbb N$. Increasing the latter if necessary, we can invoke Lemma~\ref{lem:prove-prog} to get $\sststile{0}{\alpha}\prog_{\tl}$ and then
\begin{equation*}
\AxiomC{$\sststile{d}{\alpha}\ti_{\tl}$}
\AxiomC{$\sststile{0}{\alpha}\prog_{\tl}$}
\AxiomC{$\sststile{0}{0}\neg X\overline n,X\overline n$}
\UnaryInfC{$\sststile{0}{1}\exists x.\neg Xx,X\overline n$}
\BinaryInfC{$\sststile{0}{\alpha+2}\prog_{\tl}\land\exists x.\neg Xx,X\overline n$}
\RightLabel{(cut),}
\BinaryInfC{$\sststile{d}{\alpha+3}X\overline n$}
\DisplayProof
\end{equation*}
where~$n\in\mathbb N$ is arbitrary. By the cut elimination result that is assumed in the theorem, we get $\sststile{0}{f(\alpha+3,d)}X\overline n$. Now Lemma~\ref{lem:ti-cut-free} entails that $o(n)\preceq f(\alpha+3,d)\in E$ is defined for all~$n\in\mathbb N$. Hence~$o$ witnesses $(\mathbb N,\tl)\lesssim(E\!\restriction\!f(\alpha+3,d)+1,\prec)$.
\end{proof}

In order to prove the theorem, we have assumed that the cuts in $\sststile{d}{\alpha+2}X\overline n$ can be eliminated independently of~$n$. This assumption will be justified, but it is not strictly necessary: To avoid it, derive $\sststile{d}{\alpha+2}\forall x.Xx$ by a similar argument as above. Then use a single application of cut elimination to infer $\sststile{0}{\beta}\forall x.Xx$ for some~$\beta\in E$. In the next section we will prove an `inversion' result, which allows to conclude $\sststile{0}{\beta}X\overline n$ for the same~$\beta$ as before (which is independent of~$n\in\mathbb N$). To summarize, our approach is less elementary because it requires cut elimination to be uniform. The advantage is that we avoid inversion, which we have not proved yet.

\section{Cut elimination for infinite derivations}

In this section, we prove cut elimination for the infinitary proof system introduced in Section~\ref{sect:inf-derivs}. In other words, we establish the assumption of Theorem~\ref{thm:bound-TI}, so that the latter becomes an unconditional independence result (except that it remains to specify suitable well orders $\tl$ and $\prec$). Both the general method of cut elimination~\cite{gentzen35} and its application in the ordinal analysis of Peano arithmetic~\cite{gentzen36,gentzen43} are due to Gerhard Gentzen. The formulation in terms of infinite derivations goes back to Kurt Sch\"utte~\cite{schuette-omega-rule}. 

The following notation (which Wilfried Buchholz~\cite{buchholz-local-predicativity} attributes to William Tait) will help to reduce the number of case that we need to consider:

\begin{definition}
To each $\mathcal L_{\pa}^X$-sentence $\varphi$ that is not of the form $Xt$ or $\neg Xt$ we assign a (possibly infinite) conjunction $\varphi\simeq\bigwedge_{t\in\iota(\varphi)}\varphi_t$ or disjunction $\varphi\simeq\bigvee_{t\in\iota(\varphi)}\varphi_t$. The conjunctive formulas are given as
\begin{gather*}
\varphi\simeq\text{``the empty conjunction"}\quad\text{when $\varphi$ is a true literal of~$\mathcal L_{\pa}$},\\
\varphi_0\land\varphi_1\simeq\textstyle\bigwedge_{t\in\{0,1\}}\varphi_t,\qquad
\forall x.\varphi\simeq\textstyle\bigwedge_{\text{$t$ a closed term}}\varphi[x/t].
\end{gather*}
We declare that the disjunctive formulas are the negations of the conjunctive ones. The associated disjunctions are given by $\neg\varphi\simeq\bigvee_{t\in\iota(\varphi)}\neg(\varphi_t)$ for $\varphi\simeq\bigwedge_{t\in\iota(\varphi)}\varphi_t$.
\end{definition}

It requires some practice to become familiar with the new notation:

\begin{exercise}
(a) Write out the disjunctive clauses $\varphi\simeq\bigvee_{t\in\iota(\varphi)}\varphi_t$ explicitly.

(b) Show that $\varphi\simeq\bigvee_{t\in\iota(\varphi)}\varphi_t$ entails $\neg\varphi\simeq\bigwedge_{t\in\iota(\varphi)}\neg(\varphi_t)$, or in other words: that we have $\iota(\neg\varphi)=\iota(\varphi)$ and $(\neg\varphi)_t=\neg(\varphi_t)$ for disjunctive and conjunctive~$\varphi$. \emph{Hint:}~To avoid a case distinction, recall that $\neg\neg\varphi$ and $\varphi$ are syntactically equal.
\end{exercise}

If $\sststile{d}{\alpha}\Gamma$ holds, this must be ``derived" by one of the clauses from Definition~\ref{def:inf-proofs}. We will sometimes refer to that clause as the ``last rule" in the derivation of~$\Gamma$. In the following proofs, we will distinguish cases according to this rule. Our new notation allows us to treat several cases in a uniform way. Specifically, the following is a uniform formulation of clauses~(ii,iii) and part of~(i) (the case of a true literal):
\begin{equation*}
\parbox[t]{.95\textwidth}{We get $\sststile{d}{\alpha}\Gamma$ if $\Gamma$ contains a formula $\varphi\simeq\bigwedge_{t\in\iota(\varphi)}\varphi_t$ ($\varphi\simeq\bigvee_{t\in\iota(\varphi)}\varphi_t$) and we have $\sststile{d(t)}{\alpha(t)}\Delta_t$ with $\alpha(t)\prec\alpha,d(t)\leq d$ and $\Delta_t\subseteq\Gamma,\varphi_t$ for every (some) $t\in\iota(\varphi)$.}
\end{equation*}
When $\varphi$ is a false literal, our uniform clause is void, because it requires a premise to hold for some~$t\in\iota(\varphi)=\emptyset$. The following result is a first classical ingredient of cut elimination. In a more intuitive formulation, the result says that one can transform a derivation of $\Gamma,\forall x.\psi$ into one of $\Gamma,\psi[x/t]$ for any closed term~$t$, without increasing the derivation height or the cut rank. It may help to write out the claim for a formula $\varphi\equiv\varphi_0\land\varphi_1$ in the same more intuitive fashion.

\begin{proposition}[`Inversion']\label{prop:inversion}
From $\sststile{d}{\alpha}\Gamma$ we get $\sststile{d}{\alpha}\Gamma\backslash\{\varphi\},\varphi_t$, for any conjunctive formula~$\varphi\simeq\bigwedge_{t\in\iota(\varphi)}\varphi_t$ and any $t\in\iota(\varphi)$.
\end{proposition}
\begin{proof}
We argue by induction on~$\alpha$ and distinguish cases according to the clause by which $\sststile{d}{\alpha}\Gamma$ was deduced. First assume that this clause has introduced~$\varphi\in\Gamma$, so that we have premises $\sststile{d(s)}{\alpha(s)}\Delta_s$ with $\alpha(s)\prec\alpha,d(s)\leq d$ and $\Delta_s\subseteq\Gamma,\varphi_s$ for all~$s\in\iota(\varphi)$. For $t$ as in the proposition, we inductively get~$\sststile{d(t)}{\alpha(t)}\Delta_t\backslash\{\varphi\},\varphi_t$. In~view~of
\begin{equation*}
\Delta_t\backslash\{\varphi\},\varphi_t\subseteq(\Gamma,\varphi_t)\backslash\{\varphi\},\varphi_t=\Gamma\backslash\{\varphi\},\varphi_t,
\end{equation*}
we can now use weakening to get $\sststile{d}{\alpha}\Gamma\backslash\{\varphi\},\varphi_t$. In similar arguments below, we will often assume $\Delta_s=\Gamma,\varphi_s$ and $d(s)=d$, which is permitted due to weakening. It remains to consider the cases in which $\varphi$ occurs as a side formula. Thanks to our uniform notation, there are `only' five such cases: one where $\Gamma$ contains~$Xs$ and $\neg Xt$ for terms $s,t$ with equal value; one that introduces a disjunctive formula~$\psi\in\Gamma$; one that introduces a conjunctive formula~$\psi\in\Gamma$ that is different from~$\varphi$; the progression rule; and the cut rule. We treat one of these cases and leave the others to the reader: Assume that the last rule has introduced a disjunctive formula~$\psi\in\Gamma$. Modulo weakening, it has premise $\sststile{d}{\alpha(s)}\Gamma,\psi_s$ for some $s\in\iota(\psi)$ and $\alpha(s)\prec\alpha$. We~note
\begin{equation*}
(\Gamma,\psi_s)\backslash\{\varphi\},\varphi_t\subseteq\Gamma\backslash\{\varphi\},\varphi_t,\psi_s,
\end{equation*}
where the inclusion may be strict when we have $\psi_s\equiv\varphi$. Crucially, the formulas $\psi$ and $\varphi$ are different, because no formula is both conjunctive and disjunctive. Hence we still have $\psi\in\Gamma\backslash\{\varphi\},\varphi_t$. We can thus apply the same rule to infer $\sststile{d}{\alpha}\Gamma\backslash\{\varphi\},\varphi_t$ from $\sststile{d}{\alpha(s)}(\Gamma,\psi_s)\backslash\{\varphi\},\varphi_t$, which is provided by the induction hypothesis.
\end{proof}

In the previous proof, most parts of the verifications that we have omitted are similar to the given one. Nevertheless, some points differ from case to case. For example, the reason that we have given for $\varphi\not\equiv\psi$ was specific to the case of a disjunction (cf.~the last paragraph of the proof). When writing these lecture notes, the author has checked all cases on paper, despite being familiar with the material. We would advise the reader to do the same, i.\,e., to solve part~(a) of the following exercise. Parts~(b) and~(c) will not be needed in the sequel. The result in~(b) is sometimes called inversion as well; it is also known as $\lor$-exportation.

\begin{exercise}
(a) Check all cases that have been omitted in the previous proof. Do the same for the proofs of Proposition~\ref{prop:reduction} and Theorem~\ref{thm:cut-elimination} below.

(b) Show that $\sststile{d}{\alpha}\Gamma$ entails $\sststile{d}{\alpha}\Gamma\backslash\{\varphi\lor\psi\},\varphi,\psi$.

(c) Assuming that $\psi$ is a false~$\mathcal L_{\pa}$-sentence, show that $\sststile{d}{\alpha}\Gamma$ entails $\sststile{d}{\alpha}\Gamma\backslash\{\psi\}$.
\end{exercise}

Recall that the heights of our infinite derivations are controlled by elements of a well order~$(E,\prec)$. The next steps towards cut elimination will increase the height of derivations considerably. For this reason, we need to strengthen the assumptions on our well order. As stated before, we will eventually construct a well order that satisfies all assumptions, which are thus discharged. In the following, the new binary operation~$+$ is denoted by the same symbol as the unary operation~$\alpha\mapsto\alpha+1$ from Assumption~\ref{ass:E}. In the present section, the expression $\alpha+1$ will always refer to the unary operation. Later we will construe the unary operation as a special case of the binary one (for a suitable element~$1\in E$).

\begin{assumption}\label{ass:E-omega}
Extending Assumption~\ref{ass:E}, we from now on assume that there are operations $+:E^2\to E$ and $\omega:E\to E$ with the following properties:
\begin{enumerate}[label=(\roman*)]
\item We have $\alpha+0=\alpha$ and $\alpha+\beta\prec\alpha+\gamma$ for $\beta\prec\gamma$.
\item For $\alpha\prec\beta$ we have $\omega(\alpha)\prec\omega(\beta)$. For $\alpha,\beta\prec\omega(\gamma)$ we have $\alpha+\beta\prec\omega(\gamma)$.
\end{enumerate}
We also assume that $\alpha\prec\beta$ entails $\alpha+1\preceq\beta$, and that $0\preceq\alpha$ holds for any~$\alpha\in E$.
\end{assumption}

It is common to write $\omega^\alpha$ rather than $\omega(\alpha)$, but we often use the latter to save supercripts. Concerning the following proposition, note that we cannot conclude by a cut over~$\varphi$ when we have $\rk(\varphi)=d$ (because of the condition $\rk(\varphi)<d$ in clause~(v) of Definition~\ref{def:inf-proofs}). The proof of the proposition shows how a cut over~$\varphi$ can be `reduced' to a cut over a simpler formula~$\varphi_t$. For an intuitive explanation of the reduction procedure, the reader may wish to revisit the discussion of ``cuts as lemmata", which can be found towards the end of Section~\ref{sect:sequent-calc-PL}.

\begin{proposition}[`Reduction']\label{prop:reduction}
Assume that $\varphi$ is disjunctive or of the form~$\neg Xt$, and that we have $\rk(\varphi)\leq d$. Then $\sststile{d}{\alpha}\Gamma,\neg\varphi$ and $\sststile{d}{\beta}\Gamma,\varphi$ entail $\sststile{d}{\alpha+\beta}\Gamma$.
\end{proposition}
\begin{proof}
We argue by induction on~$\beta$ and distinguish cases according to the last rule in the derivation~$\sststile{d}{\beta}\Gamma,\varphi$. First assume that $\varphi$ is a principal formula of this rule. There are two possibilities to consider: First assume that clause~(i) of Definition~\ref{def:inf-proofs} applies with $Xs\in\Gamma$ and $\neg Xt\equiv\varphi$, where $s$ and $t$ have equal value. We then have
\begin{equation*}
(\Gamma,\neg\varphi)\backslash\{Xt\},Xs\subseteq\Gamma,
\end{equation*}
as well as $\alpha=\alpha+0\preceq\alpha+\beta$. From $\sststile{d}{\alpha}\Gamma,\neg\varphi$ we can thus conclude $\sststile{d}{\alpha+\beta}\Gamma$ by Exercise~\ref{exer:inf-basic}(a) and weakening. Secondly, assume that we have $\varphi\simeq\bigvee_{t\in\iota(\varphi)}\varphi_t$ and that $\sststile{d}{\beta}\Gamma,\varphi$ was deduced from $\sststile{d}{\beta(t)}\Gamma,\varphi,\varphi_t$ with $\beta(t)\prec\beta$ for some $t\in\iota(\varphi)$. After weakening $\sststile{d}{\alpha}\Gamma,\neg\varphi$ into $\sststile{d}{\alpha}\Gamma,\neg\varphi,\varphi_t$, we can use the induction hypothesis to get
\begin{equation*}
\sststile{d}{\alpha+\beta(t)}\Gamma,\varphi_t.
\end{equation*}
As $\neg\varphi$ is conjunctive with $\iota(\neg\varphi)=\iota(\varphi)\ni t$ and $(\neg\varphi)_t\equiv\neg(\varphi_t)$, we can also infer $\sststile{d}{\alpha}\Gamma,\neg\varphi_t$ from $\sststile{d}{\alpha}\Gamma,\neg\varphi$ by Proposition~\ref{prop:inversion} (inversion). It is straightforward to see that we have $\rk(\varphi_t)<\rk(\varphi)\leq d$ as well as $\alpha\preceq\alpha+\beta(t)\prec\alpha+\beta$. We can thus conclude $\sststile{d}{\alpha+\beta}\Gamma$ by a cut over~$\varphi_t$. It remains to consider the cases where $\varphi$ occurs as a side formula. We treat the progression rule as a representative example and leave the other cases to the reader: Assume $\Gamma,\varphi$ contains a formula $Xt$ such that we have $\sststile{d}{\beta(s)}\Gamma,\varphi,Xs$ with $\beta(s)\prec\beta$ for all $s\tl t$. The induction hypothesis yields
\begin{equation*}
\sststile{d}{\alpha+\beta(s)}\Gamma,Xs.
\end{equation*}
Crucially, the formula $Xt$ is still contained in~$\Gamma$, as it is different from $\varphi$ (being neither disjunctive nor of the form~$\neg Xt$). We can thus re-apply the progression rule in order to infer $\sststile{d}{\alpha+\beta}\Gamma$.
\end{proof}

To formulate the final result on cut elimination, we iterate the map $\omega:E\to E$ from Assumption~\ref{ass:E-omega}: Given $\alpha\in E$, use recursion on~$n\in\mathbb N$ to define
\begin{equation*}
\omega(\alpha,0):=\alpha\qquad\text{and}\qquad\omega(\alpha,n+1):=\omega(\omega(\alpha,n)).
\end{equation*}
Once again, it is common to write $\omega^\alpha_n$ rather than $\omega(\alpha,n)$, but the latter allows us to save supercripts. The following result shows that the assumption of Theorem~\ref{thm:bound-TI} is satisfied (take $(\alpha,n)\mapsto\omega(\alpha,n)$ for~$f$). To obtain an unconditional independence result, it remains to discharge Assuptions~\ref{ass:E} and~\ref{ass:E-omega}, i.\,e., to construct a well order $(E,\prec)$ that satisfies these assumptions. This will be done in the next section.

\begin{theorem}[`Cut elimination']\label{thm:cut-elimination}
The proof system from Definition~\ref{def:inf-proofs} satisfies
\begin{equation*}
\sststile{d+1}{\alpha}\Gamma\quad\Rightarrow\quad\sststile{d}{\omega(\alpha)}\Gamma\quad\qquad\text{and}\qquad\quad\sststile{d}{\alpha}\Gamma\quad\Rightarrow\quad\sststile{0}{\omega(\alpha,d)}\Gamma.
\end{equation*}
\end{theorem}
\begin{proof}
It suffices to prove the first implication, from which the second follows by induction on~$d$. We argue by induction on~$\alpha$ and distinguish cases according to the last rule in the derivation $\sststile{d+1}{\alpha}\Gamma$. In the crucial case, we are concerned with a cut over an $\mathcal L_{\pa}^X$-sentence~$\varphi$ with $\rk(\varphi)<d+1$. The premises of such a cut have the form $\sststile{d+1}{\alpha(0)}\Gamma,\varphi$ and $\sststile{d+1}{\alpha(1)}\Gamma,\neg\varphi$ with $\alpha(i)\prec\alpha$, again up to weakening. By the induction hypothesis we get
\begin{equation*}
\sststile{d}{\omega(\alpha(0))}\Gamma,\varphi\qquad\text{and}\qquad\sststile{d}{\omega(\alpha(1))}\Gamma,\neg\varphi.
\end{equation*}
Either $\varphi$ or $\neg\varphi$ will be disjunctive or of the form~$\neg Xt$. In view of $\neg\neg\varphi\equiv\varphi$ we may assume that this holds for~$\varphi$. Then Proposition~\ref{prop:reduction} yields
\begin{equation*}
\sststile{d}{\omega(\alpha(1))+\omega(\alpha(0))}\Gamma.
\end{equation*}
Invoking Assumption~\ref{ass:E-omega}, we readily get $\omega(\alpha(1))+\omega(\alpha(0))\prec\omega(\alpha)$. Now $\sststile{d}{\omega(\alpha)}\Gamma$ follows by weakening. In all other cases, it is rather straightforward to reduce to the induction hypothesis. To give one example, let us assume that the last rule deduces a disjunctive formula $\varphi\simeq\bigvee_{t\in\iota(\varphi)}\varphi_t$ that lies in~$\Gamma$. For some~$t\in\iota(\varphi)$ we then have $\sststile{d+1}{\alpha(t)}\Gamma,\varphi_t$ with $\alpha(t)\prec\alpha$. Due to the induction hypothesis we get $\sststile{d}{\omega(\alpha(t))}\Gamma,\varphi_t$. In view of $\omega(\alpha(t))\prec\omega(\alpha)$, we can \mbox{re-apply} the same rule to conclude $\sststile{d}{\omega(\alpha)}\Gamma$.
\end{proof}

\section{An ordinal notation system}\label{sect:eps_0}

In this section we construct a well order~$\varepsilon_0$ such that Assumptions~\ref{ass:E} and~\ref{ass:E-omega} are satisfied with $\varepsilon_0$ at the place of~$E$. Based on the previous sections, we will be able to conclude that induction along~$\varepsilon_0$ is unprovable in $\pa[X]$ (Peano arithmetic with an uninterpreted relation symbol~$X$).

The simultaneous definition of $\varepsilon_0$ and $\prec$ in the following is easily disentangled: One can first define a larger set $T\supseteq\varepsilon_0$ by ignoring the condition $\alpha_i\preceq\alpha_{i-1}$ in (\ref{eq:eps_0-generate}). This set $T$ consists of all finite rooted trees with an order on the successors of each node. By the same clause as below, one obtains a binary relation $\prec$ on all of~$T$. Once $\prec$ is given, it is straightforward to single out $\varepsilon_0\subseteq T$ by recursion.

\begin{definition}\label{def:eps_0}
We simultaneously define a set $\varepsilon_0$ of terms (syntactic expressions) and a binary relation $\prec$ on $\varepsilon_0$. Our terms are generated by the following clause:
\begin{equation}\label{eq:eps_0-generate}\tag{$\star$}
\parbox{.75\textwidth}{Given terms $\alpha_0,\ldots,\alpha_{n-1}\in\varepsilon_0$ with $\alpha_i\preceq\alpha_{i-1}$ for $0<i<n$,\\ we add a new term $\langle\alpha_0,\ldots,\alpha_{n-1}\rangle\in\varepsilon_0$.}
\end{equation}
In particular, $n=0$ yields $\langle\rangle\in\varepsilon_0$. Also note that $\alpha\preceq\beta$ abbreviates the disjunction of $\alpha\prec\beta$ and $\alpha=\beta$, where the latter refers to syntactic equality as terms. The relation $\prec$ is recursively determined by
\begin{equation*}
\langle\alpha_0,\ldots,\alpha_{m-1}\rangle\prec\langle\beta_0,\ldots,\beta_{n-1}\rangle\quad\Leftrightarrow\quad
\begin{cases}
\text{$m<n$ and $\alpha_i=\beta_i$ for all $i<m$},\\
\parbox[t]{.4\textwidth}{or $\alpha_j\prec\beta_j$ and $\alpha_i=\beta_i$ for\\ some $j<m,n$ and all $i<j$,}
\end{cases}
\end{equation*}
i.\,e., by lexicographic comparisons.
\end{definition}

A set-theoretic interpretation of $\varepsilon_0$ in terms of Cantor normal forms will be given in Remark~\ref{rmk:eps_0-set-theoretic} below. In view of this interpretation, many authors write elements of $\varepsilon_0$ in the form $\omega^{\alpha_0}+\ldots+\omega^{\alpha_{n-1}}$ rather than $\langle\alpha_0,\ldots,\alpha_{n-1}\rangle$. We will later see that $(\varepsilon_0,\prec)$ is a well order. Part~(a) of the following exercise explains why we do not work with the larger set $T\supseteq\varepsilon_0$ that was mentioned above.

\begin{exercise}\label{ex:prec-linear}
(a) Find an infinite sequence $\alpha_0,\alpha_1,\ldots\subseteq T$ with $\alpha_{i+1}\prec\alpha_i$ for every index $i\in\mathbb N$ (cf.~the paragraph before Definition~\ref{def:eps_0}).

(b) Show that $\prec$ is a linear order on~$\varepsilon_0$ (and even on $T\supseteq\varepsilon_0$).
\end{exercise}

We now define the operations that are needed to satisfy Assumptions~\ref{ass:E} and~\ref{ass:E-omega}.

\begin{definition}\label{def:eps_0-addition}
Let $+:\varepsilon_0\times\varepsilon_0\to\varepsilon_0$ be given by
\begin{multline*}
\langle\alpha_0,\ldots,\alpha_{m-1}\rangle+\langle\beta_0,\ldots,\beta_{n-1}\rangle:=\langle\alpha_0,\ldots,\alpha_{i-1},\beta_0,\ldots,\beta_{n-1}\rangle\\
\text{with }i=\begin{cases}
m\quad\text{if $n=0$ or $\beta_0\preceq\alpha_j$ for all $j<m$},\\
\min\{j<m\,|\,\alpha_j\prec\beta_0\}\quad\text{otherwise}.
\end{cases}
\end{multline*}
In particular we have $\alpha+\langle\rangle=\alpha=\langle\rangle+\alpha$. Also, define $\omega:\varepsilon_0\to\varepsilon_0$ by $\omega(\alpha):=\langle\alpha\rangle$. Finally, abbreviate $0:=\langle\rangle$, $1:=\omega(0)$ and $\omega:=\omega(1)$ (where the context determines whether $\omega$ stands for the function on~$\varepsilon_0$ or for the given element).
\end{definition}

Note that the choice of~$i$ ensures that the result of addition lies in~$\varepsilon_0$. The following exercise is somewhat tedious but crucial and instructive:

\begin{exercise}\label{ex:eps_0-standing-ass}
(a) Show that the previous constructions satisfy Assumptions~\ref{ass:E} and~\ref{ass:E-omega} with $\varepsilon_0$ at the place of~$E$. \emph{Remark:} Interpret the unary operation $\alpha\mapsto\alpha+1$ from Assumption~\ref{ass:E} as binary addition with fixed argument $1=\omega(0)$. We will later show that $\prec$ is well founded on~$\varepsilon_0$, so you do not need to prove this here. 

(b) Show that our operation $+$ on $\varepsilon_0$ is not commutative, and that it is increasing but not strictly increasing in the first argument. Also show that $\alpha\prec\omega(\alpha)$ holds for any $\alpha\in\varepsilon_0$ (cf.~Remark~\ref{rmk:eps_0-set-theoretic} below). 
\end{exercise}

The following remark requires familiarity with basic set theory (as presented, e.\,g., in~\cite[Sections~1 and~2]{jech03}). At the same time, it is not necessary to understand the remark (or the set theory) to understand the rest of these lecture notes.

\begin{remark}\label{rmk:eps_0-set-theoretic}
In set theory, one has a well ordered class of ordinals that contains a unique element from each isomorphism class of (set-sized) well orders. The ordinals come with a minimal element~$0$ and a successor operation $\alpha\mapsto\alpha+1$. We reserve the letter $\lambda$ for limit ordinals, i.\,e., for ordinals $\lambda\neq 0$ such that $\alpha\prec\lambda$ entails~$\alpha+1\prec\lambda$. By $\omega$ we denote the smallest limit ordinal. Furthermore, the class of ordinals admits a principle of transfinite recursion, which can be used to define arithmetic operations that satisfy the recursive clauses
\begin{alignat*}{3}
\alpha+0&=\alpha,\qquad&\alpha+(\beta+1)&=(\alpha+\beta)+1,\qquad&\alpha+\lambda&=\sup\{\alpha+\beta\,|\,\beta<\lambda\},\\
\alpha\cdot 0&=0,\qquad&\alpha\cdot(\beta+1)&=\alpha\cdot\beta+\alpha,\qquad&\alpha\cdot\lambda&=\sup\{\alpha\cdot\beta\,|\,\beta<\lambda\},\\
\omega^0&=1,\qquad&\omega^{\beta+1}&=\omega^\beta\cdot\omega,\qquad&\omega^\lambda&=\sup\{\omega^\beta\,|\,\beta<\lambda\}.
\end{alignat*}
In connection with part~(b) of the previous exercise, the reader may wish to show that we get $1+\omega=\omega\prec\omega+1$. Any ordinal~$\alpha$ has a unique `Cantor normal form'
\begin{equation*}
\alpha=\omega^{\alpha_0}+\ldots+\omega^{\alpha_{n-1}}\qquad\text{with}\qquad\alpha_{n-1}\preceq\ldots\preceq\alpha_0\preceq\alpha.
\end{equation*}
We always have $\alpha\preceq\omega^\alpha$ and equality is possible. Thus we can have $\alpha_0=\alpha$ above, but only for $n=1$. One traditionally writes
\begin{equation*}
\varepsilon_0:=\min\{\alpha\,|\,\omega^\alpha=\alpha\}.
\end{equation*}
For any ordinal $\alpha\prec\varepsilon_0$ we have $\alpha_0\prec\alpha$ in the Cantor normal form as above. Using transfinite recursion, we can thus define a function $o$ that sends the ordinals below~$\varepsilon_0$ to the terms from Definition~\ref{def:eps_0}, with
\begin{equation*}
o\left(\omega^{\alpha_0}+\ldots+\omega^{\alpha_{n-1}}\right):=\langle o(\alpha_0),\ldots,o(\alpha_{n-1})\rangle
\end{equation*}
for an argument in Cantor normal form. To confirm that the terms have the required form, one needs to check that $o$ is order preserving. In fact, it is an isomorphism that respects the operations from Definition~\ref{def:eps_0-addition}. This justifies the notation $\varepsilon_0$ for our system of terms.
\end{remark}

Fix a reasonable coding of finite objects by natural numbers, so that basic properties can be developed in~$\pa$ (much weaker theories suffice; see e.\,g.\ \cite[Section~I.1]{hajek91}). In particular, this yields a code $c(\alpha)\in\mathbb N$ for each finite tree~$\alpha\in\varepsilon_0$. We will also write $\varepsilon_0$ and $\prec$ for the set $\{c(\alpha)\,|\,\alpha\in\varepsilon_0\}$ and the relation $\{(c(\alpha),c(\beta))\,|\,\alpha\prec\beta\}$. With this convention in place, one should tackle the following task:

\begin{exercise}\label{ex:eps-pr}
Convince yourself that $\varepsilon_0\subseteq\mathbb N$ and ${\prec}\subseteq\mathbb N^2$ are primitive recursive. \emph{Hint:} Consider the paragraph before Definition~\ref{def:eps_0}.
\end{exercise}

It is known that primitive recursive relations can be represented in $\pa$ (see again \cite{hajek91} or also \cite{smorynski-incompleteness}). We will also write $\alpha\in\varepsilon_0$ and $\alpha\prec\beta$ for fixed formulas (with free variables $\alpha,\beta$) that define $\varepsilon_0\subseteq\mathbb N$ and ${\prec}\subseteq\mathbb N^2$. As in Section~\ref{sect:inf-derivs}, the principle of transfinite induction along~$\varepsilon_0$ can now be expressed by
\begin{equation*}
\ti_{\prec}\,:\equiv\,\prog_{\prec}\to\forall\alpha\in\varepsilon_0.X\alpha\quad\text{with}\quad\prog_{\prec}\,:\equiv\,\forall\alpha\in\varepsilon_0(\forall\beta\prec\alpha.X\beta\to X\alpha).
\end{equation*}
By combining all our previous work, we can finally derive the following theorem of Gerhard Gentzen~\cite{gentzen43}. In the next section, we will use this theorem to show that Kruskal's theorem is unprovable in conservative extensions of Peano arithmetic.

\begin{theorem}\label{thm:unprov-eps_0}
The theory~$\pa[X]$ does not prove $\ti_{\prec}$.
\end{theorem}
\begin{proof}
The order $(\varepsilon_0,\prec)$ is well founded, but we have not proved this fact yet (in fact a set-theoretic proof is implicit in Remark~\ref{rmk:eps_0-set-theoretic}). To circumvent this issue, we begin with a trick: Assume that $\prec$ fails to be well founded. As in Exercise~\ref{ex:ti-wf}, this entails that $\ti_{\prec}$ is false in some standard model of $\mathcal L_{\pa}^X$. But then $\ti_{\prec}$ is unprovable in $\pa[X]$, since the latter is sound for such models. In the rest of this proof, we may thus assume that $\prec$ is a well order. We would like to conclude by Theorem~\ref{thm:bound-TI}. However, this theorem was formulated for a well order $\tl$ with field~$\mathbb N$. This was convenient in Section~\ref{sect:inf-derivs}, but it seems too restrictive here: we do not want to assume~$\varepsilon_0=\mathbb N$, since this would heavily depend on the coding. Instead, consider the unique function $e:\mathbb N\to\mathbb N$ that is strictly increasing with image $\varepsilon_0\subseteq\mathbb N$. This function is primitive recursive by bounded search. In particular, it is definable in~$\mathcal L_{\pa}$. We now fix an $\mathcal L_{\pa}$-formula $x\tl y$ that defines the relation ${\tl}\subseteq\mathbb N^2$ with
\begin{equation*}
m\tl n\quad\Leftrightarrow\quad e(m)\prec e(n).
\end{equation*}
Given that $\prec$ is a well order on~$\varepsilon_0$ (cf.~Exercise~\ref{ex:prec-linear} and the beginning of this proof), it is immediate that $\tl$ is a well order with field~$\mathbb N$, as required in Section~\ref{sect:inf-derivs}. Let us check the remaining assumptions of Theorem~\ref{thm:bound-TI}: From Exercise~\ref{ex:eps_0-standing-ass} we know that Assumptions~\ref{ass:E} and~\ref{ass:E-omega} hold with $\varepsilon_0$ at the place of~$E$. Under these assumptions, Theorem~\ref{thm:cut-elimination} shows that our infinite proof system admits cut elimination. Furthermore, we have $(\varepsilon_0,\prec)\cong(\mathbb N,\tl)$ by the definition of~$\tl$. Now that we have checked all assumptions, Theorem~\ref{thm:bound-TI} says that $\ti_{\tl}$ is unprovable in $\pa[X]$. To conclude that there is no proof of $\ti_{\prec}$ in $\pa[X]$, it suffices to show that such a proof could be transformed into one of~$\ti_{\tl}$. Crucially, a proof in $\pa[X]$ remains valid when we replace each occurrence $Xt$ of our relation symbol by the corresponding instance $\varphi(t)$ of some $\mathcal L_{\pa}^X$-formula~$\varphi$. Indeed, the equality axiom for $X$ is transformed into the formula $x=y\land\varphi(x)\to\varphi(y)$, which is provable; and any induction axiom remains an axiom, since the induction formula may be arbitrary. Assuming that we have a proof of $\ti_\prec$, we apply this observation with
\begin{equation*}
\varphi(x):\equiv Xe^{-1}(x):\equiv\exists y(e(y)=x\land Xy).
\end{equation*}
This yields a proof of the formula
\begin{equation*}
\forall\alpha\in\varepsilon_0(\forall\beta\prec\alpha.Xe^{-1}(\beta)\to Xe^{-1}(\alpha))\to\forall\alpha\in\varepsilon_0.Xe^{-1}(\alpha).
\end{equation*}
Above we have seen that $e$ is definable in~$\mathcal L_{\pa}$. For a reasonable definition of~$e$, Peano arithmetic will show that $e$ is an order embedding from $(\mathbb N,\tl)$ onto $(\varepsilon_0,\prec)$ (see \cite[Section~I.1]{hajek91} for the formalization of basic arguments in~$\pa$). Now it is straightforward to complete the transformation into a proof of $\ti_{\tl}$. But we have seen that the latter is unprovable in~$\pa[X]$. Hence the same must hold for~$\ti_{\prec}$.
\end{proof}

Gentzen~\cite{gentzen43} has also shown that his theorem is sharp, in the sense that $\pa[X]$ proves transfinite induction along any proper initial segment of~$\varepsilon_0$. As we will not need this result for our independence proof, we present it in the form of the exercise below (see e.\,g.\ \cite[Section~4]{sommer95} for a detailed solution). To summarize Gentzen's results, one often says that $\varepsilon_0$ is the proof-theoretic ordinal of Peano arithmetic. A precise definition of \emph{proof-theoretic ordinal} requires some care. There are several different definitions, which emphasize different aspects of ordinal analysis and are often but not always equivalent (see~\cite{rathjen-realm} and~\cite[Section~6]{wainer-fairtlough-98}). The author of these lecture notes thinks that these issues should not be overstated: no general definition of proof-theoretic ordinal is needed for concrete applications of ordinal analysis, e.\,g., for independence proofs.

\begin{exercise}\label{ex:Gentzen-jump}
For an $\mathcal L_{\pa}^X$-formula~$\varphi(\alpha)$ with a distinguished free variable, write
\begin{align*}
\prog_{\prec}(\varphi)\,&:\equiv\,\forall\alpha\in\varepsilon_0(\forall\beta\prec\alpha.\varphi(\beta)\to\varphi(\alpha)),\\
\ti_{\prec}(\varphi,\gamma)\,&:\equiv\,\prog_{\prec}(\varphi)\to\forall\alpha\prec\gamma.\varphi(\alpha).
\end{align*}
Here $\gamma\in\varepsilon_0$ is a fixed ordinal notation, to which we refer via its numeral (so it would be more correct to write $\alpha\prec\overline\gamma$ rather than $\alpha\prec\gamma$). A crucial idea of Gentzen was to define the `jump' of~$\varphi$ as the formula
\begin{equation*}
\varphi^J(\alpha)\,:\equiv\,\forall\beta\in\varepsilon_0\left(\forall\gamma\prec\beta.\varphi(\gamma)\to\forall\gamma\prec\beta+\omega(\alpha).\varphi(\gamma)\right).
\end{equation*}
Note that this involves $\mathcal L_{\pa}$-definitions of the operations $+$ and $\omega$ that were introduced in Definition~\ref{def:eps_0-addition}. Basic properties of these operations are provable in~$\pa$. A list of such properties can be found in \cite[Section3.5]{sommer95}, but you can also try to identify and establish them yourself. Show that $\pa[X]$ proves
\begin{enumerate}[label=(\alph*)]
\item the implication $\prog_{\prec}(\varphi)\to\prog_{\prec}(\varphi^J)$,
\item the implication $\ti_{\prec}(\varphi^J,\gamma)\to\ti_{\prec}(\varphi,\omega(\gamma))$,
\item the statement $\ti_{\prec}(\varphi,\gamma)$ for each $\mathcal L_{\pa}^X$-formula~$\varphi(\alpha)$ and each $\gamma\prec\varepsilon_0$.
\end{enumerate}
To establish~(c), use induction on~$n\in\mathbb N$ (in the meta theory) to show the claim for $\gamma=\omega_n$ with $\omega_0=1$ and $\omega_{n+1}=\omega(\omega_n)$. Then argue that any $\gamma\in\varepsilon_0$ lies below some~$\omega_n$. To avoid misunderstanding, we stress that the quantification over $\gamma$ in~(c) takes place in the meta language. Indeed, Theorem~\ref{thm:unprov-eps_0} shows that $\pa[X]$ does not prove the statement $\forall\gamma\prec\varepsilon_0.\ti_{\prec}(\varphi,\gamma)$, in which the quantification over~$\gamma$ is internalized. For this reason, it must be impossible to internalize the induction on~$n$ that we use to prove~(c). To see why this is the case, note that the jump $\varphi^J$ has higher quantifier complexity than~$\varphi$. In the inductive proof of~(c), we take iterated jumps. This means that we consider induction formulas of higher and higher quantifier complexity. It is not possible to represent all these induction formulas by a single formula of~$\mathcal L_{\pa}^X$. To conclude this exercise, we draw one further conclusion in the meta theory: from~(c) we learn that $\ti_{\prec}(\varphi,\gamma)$ holds in all standard models of $\mathcal L_{\pa}^X$ (cf.~Exercise~\ref{ex:ti-wf}), for any $\gamma\in\varepsilon_0$. As in the cited exercise, $\prec$ is thus well founded on any initial segment $\{\alpha\in\varepsilon_0\,|\,\alpha\prec\gamma\}$ of $\varepsilon_0$ (take $\varphi(\alpha)=X\alpha$). One can conclude that the entire order $(\varepsilon_0,\prec)$ is well founded (even though this cannot be proved in~$\pa[X]$). We will see another proof of this fact in the next section.
\end{exercise}

\section{Unprovability of Kruskal's theorem}\label{sect:kruskal}

In this section, we use Theorem~\ref{thm:unprov-eps_0} to derive that Kruskal's theorem (even restricted to binary trees) is unprovable in (conservative extensions of) Peano arithmetic.

The relevant version of Kruskal's theorem has been stated as Theorem~\ref{thm:kruskal-binary} in the introduction. To make the statement of this theorem precise, we now define the set $\mathcal B$ and the relation $\leq_{\mathcal B}$ to which it refers:

\begin{definition}\label{def:binary-trees}
The set $\mathcal B$ consists of the syntactic expressions generated as follows:
\begin{enumerate}[label=(\roman*)]
\item We have an element $\circ\in\mathcal B$.
\item Given $s,t\in\mathcal B$, we add an element $\circ(s,t)\in\mathcal B$.
\end{enumerate}
To define a binary relation $\leq_{\mathcal B}$ on $\mathcal B$ by recursion, we stipulate that we have
\begin{equation*}
\circ(s_0,s_1)\leq_{\mathcal B}\circ(t_0,t_1)\quad\Leftrightarrow\quad
\begin{cases}
\text{we have $s_0\leq_{\mathcal B}t_0$ and $s_1\leq_{\mathcal B}t_1$},\\
\text{or we have $\circ(s_0,s_1)\leq_{\mathcal B}t_i$ for some $i\in\{0,1\}$},
\end{cases}
\end{equation*}
and that we always have $\circ\leq_{\mathcal B}t$ while $s\leq_{\mathcal B}\circ$ holds for $s=\circ$ only.
\end{definition}

Intuitively, the elements of $\mathcal B$ correspond to finite binary trees with root, in which every vertex has zero or two successors. The definition of $\leq_{\mathcal B}$ relies on the fact that left and right successor are distinguished. If we want to ``forget" the order between successors, we can consider the relation $\leq_{\mathcal B}^-$ that is defined as $\leq_{\mathcal B}$ but with the additional clause that $\circ(s_0,s_1)\leq_{\mathcal B}^-\circ(t_0,t_1)$ does also hold when we have $s_0\leq_{\mathcal B}^-t_1$ and $s_1\leq_{\mathcal B}^-t_0$. Write $s=_{\mathcal B}t$ to denote the conjunction of $s\leq_{\mathcal B}^-t$ and $t\leq_{\mathcal B}^-s$. The quotient $\mathcal B/=_{\mathcal B}$ by this equivalence relation corresponds to a notion of binary trees that does not distinguish between left and right successor. Our main results remain valid with $(\mathcal B/=_{\mathcal B},\leq_{\mathcal B}^-)$ at the place of $(\mathcal B,\leq_{\mathcal B})$, as Floris van Vugt has shown in an honours thesis~\cite{van-vugt} supervised by Andreas Weiermann. We will only prove them for $(\mathcal B,\leq_{\mathcal B})$, which is considerably easier. Part~(a) of the following exercise connects our definition to terminology from graph theory.

\begin{exercise}
(a) For $s\in\mathcal B$, let $V(s)$ be the set of vertices of the corresponding tree. Show that each verification of $s\leq_{\mathcal B}t$ according to the recursive definition of~$\leq_{\mathcal B}$ corresponds to a function $f:V(s)\to V(t)$ that preserves infima with respect to the usual tree order. Conclude that $s$ is a topological minor of~$t$ when $s\leq_{\mathcal B}t$.

(b) Show that $\leq_{\mathcal B}$ is a partial order on~$\mathcal B$. \emph{Hint:} For proofs by induction, use the height function $h:\mathcal B\to\mathbb N$ with $h(\circ)=0$ and $h(\circ(s,t))=\max\{h(s),h(t)\}+1$.
\end{exercise}

Let us introduce some terminology that puts Kruskal's theorem into context: Given a partial order~$(X,\leq_X)$, a finite or infinite sequence $x_0,x_1,\ldots\subseteq X$ is called bad if $x_i\leq_X x_j$ holds for no $i<j$. One calls $(X,\leq_X)$ a \emph{well partial order} if all bad sequences are finite. In this terminology, Theorem~\ref{thm:kruskal-binary} asserts that $(\mathcal B,\leq_{\mathcal B})$ is a well partial order. Given that Kruskal's theorem is a result of combinatorics rather than logic, we present its proof in the form of an exercise. Part~(c) of the exercise provides additional context that will not be needed in the rest of these lecture notes.

\begin{exercise}\label{ex:Kruskal}
(a) For an infinite sequence $x_0,x_1,\ldots$ in a well partial order $(X,\leq_X)$, show that there is an infinite $A\subseteq\mathbb N$ with $x_i\leq_X x_j$ for all $i<j$ with $i,j\in A$. \emph{Hint:}~Use the infinite Ramsey theorem (for pairs and two colours), which states that any function
\begin{equation*}
c:[X]^2=\left\{\{x,y\}\,|\,x\neq y\in X\right\}\to\{0,1\}
\end{equation*}
is constant on $[Z]^2$ for some infinite $Z\subseteq X$ (if $X$ itself is infinite).

(b) Prove Theorem~\ref{thm:kruskal-binary}. \emph{Hint:} Assuming that the theorem fails, construct a bad sequence $r_0,r_1,\ldots\subseteq\mathcal B$ such that no sequence $r_0,\ldots,r_{i-1},r_i',r_{i+1}',\ldots$ with $h(r_i')<h(r_i)$ is bad (with $h$ as in the previous exercise). Writing $r_i=\circ(s_i,t_i)$, conclude that $\leq_{\mathcal B}$ is a well partial order on $\{s_i\,|\,i\in\mathbb N\}$ and on $\{t_i\,|\,i\in\mathbb N\}$. Use~(a) to infer that $r_0,r_1,\ldots$ cannot be bad after all. This approach is due to Crispin Nash-Williams~\cite{nash-williams63} and known as the \emph{minimal bad sequence argument}. A full solution of the exercise is implicit in the cited paper (which is very short and readable).

(c) Given a partial order~$(X,\leq_X)$, show that the following are equivalent:
\begin{itemize}
\item $(X,\leq_X)$ is a well partial order,
\item any $Y\subseteq X$ has a finite basis (i.\,e., there is a finite $Y_0\subseteq Y$ with the following property: for any $y\in Y$ there is a $y_0\in Y_0$ with $y_0\leq_X y$),
\item $(X,\leq_X)$ contains no infinite descending sequence and no infinite antichain (where an antichain is an infinite set of pairwise incomparable elements),
\item any linear extension of $(X,\leq_X)$ is a well order (where a linear extension is a linear order~$(X,\preceq)$ such that $x\leq_X y$ entails $x\preceq y$).
\end{itemize}
A solution for part~(c) can be found in~\cite[Section~2]{marcone-survey-wqo-bqo}. In addition to proving them, the cited paper gives a logical analysis of these equivalences (and of many other results about well partial orders).
\end{exercise}

Modulo a standard encoding of finite objects, we can view $\mathcal B$ as a subset of~$\mathbb N$. This subset and the resulting relation ${\leq_{\mathcal B}}\subseteq\mathcal B^2\subseteq\mathbb N^2$ are primitive recursive and, in particular, definable in~$\mathcal L_{\pa}$ (cf.~Exercise~\ref{ex:eps-pr} and the paragraph that precedes it). We could now express Theorem~\ref{thm:kruskal-binary} in the language $\mathcal L_{\pa}^X$ (take $X=\{\langle i,t_i\rangle\,|\,i\in\mathbb N\}$ with $\langle i,t_i\rangle\in\mathbb N$ via Cantor pairing). However, some readers may find $\mathcal L_{\pa}^X$ and $\pa[X]$ somewhat unsatisfactory, since the access to infinite objects is very limited. For this reason, we shall now present another conservative extension of $\pa$, in which infinite subsets of $\mathbb N$ are ``first-class citizens".

A model of second order arithmetic is a pair $(\mathcal N,\mathcal S)$ of an $\mathcal L_{\pa}$-structure~$\mathcal N$ and a subset $\mathcal S\subseteq\mathcal P(\mathcal N)$ of the powerset of~$\mathcal N$. On the syntactic side, second order arithmetic involves two sorts of variables: the first and second order variables (also called number and set variables) are denoted by lower and upper case letters $x,y,\ldots$ and $X,Y,\ldots$ and range over $\mathcal N$ and~$\mathcal S$, respectively. In addition to the prime formulas of $\mathcal L_\pa$, we have a prime formula $t\in X$ for each second order variable~$X$ and each term $t$ of the first order language~$\mathcal L_\pa$. To obtain the formulas of second order arithmetic (also called $\mathcal L_2$-formulas), one combines the prime formulas by the usual connectives ($\neg,\to$~etc.), quantifiers over number variables (written $\forall x$ and $\exists x$ or also $\forall x\in\mathbb N$ and $\exists x\in\mathbb N$) and quantifiers over set variables (written $\forall X$ and $\exists X$ or $\forall X\subseteq\mathbb N$ and~$\exists X\subseteq\mathbb N$). One can construe these formulas as first order formulas in a language with predicate symbols for the two sorts (see e.\,g.~\cite[Remark~III.1.15]{hajek91}), and completeness holds with respect to the aforementioned notion of model.

Consider an $\mathcal L_2$-formula $\varphi(x)$ with a distinguished free number variable (and possibly with further free number and set variables as parameters). The comprehension principle for $\varphi$ is (the universal closure of) the formula
\begin{equation}\tag{$\varphi\textsf{-CA}$}
\exists X\forall x(x\in X\leftrightarrow\varphi(x)),
\end{equation}
which intuitively asserts the existence of $X=\{x\in\mathbb N\,|\varphi(x)\}$. An $\mathcal L_2$-formula is called \emph{arithmetical} if it contains no quantifiers over set variables (but free set variables are allowed). We will see that Peano arithmetic is closely related to the \mbox{$\mathcal L_2$-theory} $\aca_0$ (arithmetical comprehension) that has the following axioms: First, we keep the axioms of Robinson arithmetic and the equality axioms, now including the instance $x=y\land x\in X\to y\in X$ (we do not include equality between sets but abbreviate $\forall x(Xx\leftrightarrow Yx)$ by $X=Y$). Secondly, we include the comprehension principle $\varphi\textsf{-CA}$ for each arithmetical~$\varphi$. Finally, we admit induction in the form
\begin{equation*}
0\in X\land\forall x(x\in X\to x+1\in X)\to\forall x.x\in X,
\end{equation*}
where $X$ is a set variable. It is worth working out the following in some detail:

\begin{exercise}\label{ex:aca_0-ind}
Show that $\aca_0$ proves the induction statement $\mathcal I\varphi$ for each arithmetical $\mathcal L_2$-formula $\varphi$ (see Section~\ref{sect:inf-derivs} for the notation). \emph{Hint:} To reduce to induction for set variables, form $X=\{x\in\mathbb N\,|\,\varphi(x)\}$ by arithmetical comprehension.
\end{exercise}

In contrast, if we add an axiom $\mathcal I\varphi$ for every $\mathcal L_2$-formula $\varphi$, we get a strictly stronger theory denoted by $\aca$ (without the subscript zero; e.\,g., $\aca$ proves the full infinite Ramsey theorem while $\aca_0$ does not, as shown by Carl Jockusch~\cite{jockusch-ramsey}). Even stronger theories arise if we admit comprehension for formulas with set quantifiers. These theories and their relation to mathematical practice are studied in the research programme of \emph{reverse mathematics}. More specifically, a typical result of reverse mathematics will show that some mathematical theorem is equivalent to a comprehension or other set existence principle, provably in a weak base theory. We note that this base theory should not prove the theorems under consideration, or else it will prove the equivalence for trivial reasons. For more information on reverse mathematics we refer to the founding paper by Harvey Friedman~\cite{friedman-rm} and the textbook by Stephen Simpson~\cite{simpson09}.

We identify each $\mathcal L_{\pa}^X$-formula with the $\mathcal L_2$-formula in which each occurrence $Xt$ of the predicate symbol $X$ is replaced by $t\in X$ for a fixed set variable that we denote by~$X$ as well. Modulo this identification, we have the following connection with Peano arithmetic (the given proof follows~\cite{hajek91}):

\begin{proposition}\label{prop:aca-conservative}
The theory $\aca_0$ is a conservative extension of $\pa[X]$, i.\,e., the two theories prove the same $\mathcal L_{\pa}^X$-formulas.
\end{proposition}
\begin{proof}
In order to see that $\pa[X]$ is contained in $\aca_0$, note that the latter contains all $\pa[X]$-axioms except for the induction statements~$\mathcal I\varphi$. The latter are provable whenever $\varphi$ is an $\mathcal L_{\pa}^X$-formula and hence arithmetical, by Exercise~\ref{ex:aca_0-ind}. For conservativity, we show that any model $(\mathcal N,X^{\mathcal N})$ of $\pa[X]$ can be extended into a model $(\mathcal N,\mathcal S)$ of $\aca_0$ with $X^\mathcal N\in\mathcal S$ (where $X^\mathcal N$ is the interpretation of the relation symbol $X$ from~$\mathcal L_{\pa}^X$). Once this is established, we can conclude as follows: Assume that a given $\mathcal L_{\pa}^X$-formula~$\psi$ is unprovable in~$\pa[X]$. By completeness (for first order logic) we get a model $(\mathcal N,X^{\mathcal N})\nvDash\psi$. If we view $\psi$ as an $\mathcal L_2$-formula, this means that $\psi$ fails in $\mathcal N$ when the free set variable~$X$ is interpreted as $X^{\mathcal N}$. In particular, $\psi$ is not valid in our model $(\mathcal N,\mathcal S)\vDash\aca_0$ with $X^{\mathcal N}\in\mathcal S$. Now soundness (for second order arithmetic) entails that $\psi$ is unprovable in $\aca_0$. By contraposition, if $\aca_0$ proves $\psi$, then so must $\pa[X]$, as conservativity demands. It remains to describe the required transformation of models. Given $(\mathcal N,X^{\mathcal N})\vDash\pa[X]$, we put
\begin{gather*}
\mathcal S=\{\llbracket\varphi(x,\mathbf n)\rrbracket\,|\,\varphi(x,\mathbf y)\text{ an $\mathcal L_{\pa}^X$-formula and }\mathbf n\in\mathcal N\}\\
\text{with}\quad\llbracket\varphi(x,\mathbf n)\rrbracket=\{m\in\mathcal N\,|\,(\mathcal N,X^{\mathcal N})\vDash\varphi(m,\mathbf n)\}.
\end{gather*}
Here $\mathbf y$ stands for a sequence of variables, which we often suppress for notational convenience. Let us observe $X^{\mathcal N}=\llbracket Xx\rrbracket\in\mathcal S$. It remains to show $(\mathcal N,\mathcal S)\vDash\aca_0$. The axioms of Robinson arithmetic hold in $(\mathcal N,\mathcal S)$ since they are contained in $\pa[X]$ and do not depend on~$\mathcal S$. Concerning equality, we may assume that $=$ is interpreted by actual equality in~$\mathcal N$. Arithmetical comprehension amounts to the following: for any arithmetical $\mathcal L_2$-formula $\psi(x,Z_1,\ldots,Z_n)$ and any $S_1,\ldots,S_k\in\mathcal S$ (as well as suppressed number parameters from~$\mathcal N$), there is a $T\in\mathcal S$ with
\begin{equation*}
\forall m\in\mathcal N(m\in T\leftrightarrow(\mathcal N,\mathcal S)\vDash\psi(m,S_1,\ldots,S_k)).
\end{equation*}
Note that this statement does not actually depend on~$\mathcal S$, as $\psi$ contains no set quantifiers. In view of $S_i\in\mathcal S$ we can write $S_i=\llbracket\varphi_i(x)\rrbracket$. Let $\varphi$ be the $\mathcal L_{\pa}^X$-formula that results from $\psi(x,Z_1,\ldots,Z_k)$ when we replace each occurrence of $t\in Z_i$ by the formula~$\varphi_i(t)$. We then have
\begin{equation*}
(\mathcal N,\mathcal S)\vDash\psi(m,S_1,\ldots,S_k)\quad\Leftrightarrow\quad(\mathcal N,X^{\mathcal N})\vDash\varphi(m),
\end{equation*}
by a straightforward induction over~$\psi$. Thus it suffices to set $T=\llbracket\varphi(x)\rrbracket$. Finally, the induction axiom of $\aca_0$ requires that we have
\begin{equation*}
0\in\llbracket\varphi(x)\rrbracket\land\forall m\in\mathcal N(m\in\llbracket\varphi(x)\rrbracket\to m+1\in\llbracket\varphi(x)\rrbracket)\to\forall m\in\mathcal N.m\in\llbracket\varphi(x)\rrbracket
\end{equation*}
for any intepretation $\llbracket\varphi(x)\rrbracket\in\mathcal S$ of the set variable (with $0,1$ and $+$ interpreted as in~$\mathcal N$). This amounts to $(\mathcal N,X^{\mathcal N})\vDash\mathcal I\varphi$, which holds since $\mathcal I\varphi$ is a $\pa[X]$-axiom.
\end{proof}

We assume that the collection $\mathcal B$ of finite binary trees and the embeddability relation $\leq_{\mathcal B}$ are represented in $\pa\subseteq\aca_0$, as discussed above.  In $\aca_0$ we can use set variables to represent quantification over infinite sequences $t_0,t_1,\ldots\subseteq\mathcal B$, using the Cantor pairing function $\langle\cdot,\cdot\rangle$. Thus Kruskal's theorem for binary trees (i.\,e., Theorem~\ref{thm:kruskal-binary} from the introduction) can be expressed by the ``pseudo $\mathcal L_2$-formula"
\begin{multline*}
\forall X\subseteq\mathbb N[\forall i\in\mathbb N\exists t,p\in\mathbb N(t\in\mathcal B\land p=\langle i,t\rangle\land p\in X)\to\\
\exists{i,j,s,t,p,q\in\mathbb N}(i<j\land s,t\in\mathcal B\land p=\langle i,s\rangle\land q=\langle j,t\rangle\land p,q\in X\land s\leq_{\mathcal B}t)].
\end{multline*}
Some variation is possible: In particular, we could strengthen the antecedent by demanding that all elements of $X$ have the form $\langle i,t\rangle$ with $t\in\mathcal B$ and that there is only one such $t$ for each~$i$. However, the resulting formula will be equivalent (provably in~$\aca_0$). We have called the above a ``pseudo formula", since it involves expressions that are not in the language $\mathcal L_2$ (analogous to the more common term ``pseudocode" in the context of computer science). Each such expression can be replaced by an actual formula (e.\,g., the relation $z=\langle x,y\rangle$ has a reasonable representation in~$\pa$ since Cantor pairing is primitive recursive). It takes some practice to use pseudo formulas with confidence. The interested reader is advised to work through the first sections of~\cite{hajek91} (for a very thorough treatment of first order arithmetic) and~\cite{simpson09} (for second order aspects).

The following result will help us to connect Kruskal's theorem to the well order $(\varepsilon_0,\prec)$ that was described in Section~\ref{sect:eps_0}. The formalization of the claim in $\mathcal L_2$ is similar to the one of Kruskal's theorem above. In particular, a set variable is used to assert that $\{\langle\alpha,f(\alpha)\rangle\,|\,\alpha\in\varepsilon_0\}\subseteq\mathbb N$ exists.

\begin{proposition}\label{prop:quasi-embedding}
Provably in~$\aca_0$ (in fact in weaker theories), there is a function $f:\varepsilon_0\to\mathcal B$ that reflects the order, i.\,e., such that $f(\alpha)\leq_{\mathcal B}f(\beta)$ entails $\alpha\preceq\beta$.
\end{proposition}
\begin{proof}
First define an auxiliary ``length" function $l:\varepsilon_0\to\mathbb N$ by setting
\begin{equation*}
l(\langle\alpha_0,\ldots,\alpha_{n-1}\rangle):=l(\alpha_0)+\ldots+l(\alpha_{n-1})+n.
\end{equation*}
In particular, we take this to mean $l(\langle\rangle)=0$ for $n=0$. For $\alpha\in\varepsilon_0$ we now use recursion over $l(\alpha)$ to define~$f(\alpha)$, by setting
\begin{equation*}
f(\langle\alpha_0,\ldots,\alpha_{n-1}\rangle):=\begin{cases}
\circ & \text{if $n=0$},\\
\circ(f(\alpha_0),f(\langle\alpha_1,\ldots,\alpha_{n-1}\rangle)) & \text{otherwise}.
\end{cases}
\end{equation*}
Concerning the first case, recall that $\circ$ corresponds to the tree that consists of a root only. In the second case, the root has two recursively given subtrees. The function~$f$ is primitive recursive, e.\,g., because the number of $\alpha\in\varepsilon_0$ with $l(\alpha)\leq n$ is finite and effectively bounded for each~$n\in\mathbb N$. In particular, $f$ has an arithmetical definition for which $\aca_0$ proves the defining equation (see e.\,g.~\cite[Section~I.1(c)]{hajek91}). We use induction over $l(\alpha)+l(\beta)$ to prove
\begin{equation*}
f(\alpha)\leq_{\mathcal B}f(\beta)\quad\Rightarrow\quad\alpha\preceq\beta.
\end{equation*}
Note that the induction can be carried out in $\aca_0$, since the claim is arithmetical (cf.~Exercise~\ref{ex:aca_0-ind}). Let us first observe that the desired implication holds for $\alpha=\langle\rangle$ (as $\langle\rangle\preceq\beta$ holds for any $\beta$) and for $\beta=\langle\rangle$ (where $f(\alpha)\leq_{\mathcal B}f(\beta)=\circ$ entails $f(\alpha)=\circ$ and hence $\alpha=\langle\rangle$). In the remaining case, we can write $\alpha=\langle\alpha_0,\ldots,\alpha_{m-1}\rangle$ and $\beta=\langle\beta_0,\ldots,\beta_{n-1}\rangle$ with~$m,n>0$. Let us abbreviate $\alpha'=\langle\alpha_1,\ldots,\alpha_{m-1}\rangle$ as well as $\beta'=\langle\beta_1,\ldots,\beta_{n-1}\rangle$. We note that this yields $\alpha=\omega(\alpha_0)+\alpha'$ and $\beta=\omega(\beta_0)+\beta'$ according to Definition~\ref{def:eps_0-addition}. In view of Definition~\ref{def:binary-trees}, the inequality
\begin{equation*}
f(\alpha)=\circ(f(\alpha_0),f(\alpha'))\leq_{\mathcal B}\circ(f(\beta_0),f(\beta'))=f(\beta)
\end{equation*}
can hold for two different reasons: First, assume that we have $f(\alpha_0)\leq_{\mathcal B}f(\beta_0)$ as well as $f(\alpha')\leq_{\mathcal B}f(\beta')$. By induction hypothesis, this entails $\alpha_0\preceq\beta_0$ and $\alpha'\preceq\beta'$. The result of Exercise~\ref{ex:eps_0-standing-ass} yields $\omega(\alpha_0)\preceq\omega(\beta_0)$ and then
\begin{equation*}
\alpha=\omega(\alpha_0)+\alpha'\preceq\omega(\beta_0)+\alpha'\preceq\omega(\beta_0)+\beta'=\beta.
\end{equation*}
We point out that Exercise~\ref{ex:eps_0-standing-ass} is solved by an induction over terms, which can be formalized in $\pa\subseteq\aca_0$ (and in much weaker theories, see e.\,g.~\cite[Section~3]{sommer95}). In the remaining case, the given inequality $f(\alpha)\leq_{\mathcal B}f(\beta)$ holds because we have $f(\alpha)\leq_{\mathcal B}f(\beta_0)$ or $f(\alpha)\leq_{\mathcal B}f(\beta')$. Inductively, this yields $\alpha\preceq\beta_0$ or $\alpha\preceq\beta'$. Again by Exercise~\ref{ex:eps_0-standing-ass}, we have $\beta_0\prec\omega(\beta_0)+\langle\rangle\preceq\beta$ as well as $\beta'\preceq\beta$. We now get $\alpha\preceq\beta$ by the transitivity of~$\prec$ (see Exercise~\ref{ex:prec-linear}).
\end{proof}

It may be interesting to observe that injectivity is automatic:

\begin{exercise}
Show that any order reflecting $f:P\to Q$ between partial orders is injective. \emph{Remark:} An order reflecting function is also called a \emph{quasi embedding}. 
\end{exercise}

Finally, we prove Theorem~\ref{thm:kruskal-unprovable} from the introduction, for the given formalization of Kruskal's theorem and the conservative extension~$\aca_0$ of Peano arithmetic. Note that Kruskal's theorem for arbitrary finite (not necessarily binary) trees is even stronger and hence also unprovable in~$\aca_0$.

\begin{theorem}\label{thm:Kruskal-indep-ACA}
The theory $\aca_0$ does not prove Kruskal's theorem for binary trees (i.\,e., it does not prove Theorem~\ref{thm:kruskal-binary} with $\mathcal B$ and $\leq_{\mathcal B}$ as in Definition~\ref{def:binary-trees}).
\end{theorem}
\begin{proof}
By combining Theorem~\ref{thm:unprov-eps_0} with Proposition~\ref{prop:aca-conservative}, we learn that $\aca_0$ does not prove $\ti_{\prec}$, which expresses induction along the well order $(\varepsilon_0,\prec)$. It only remains to show that $\ti_{\prec}$ follows from Kruskal's theorem for binary trees, provably in~$\aca_0$. Aiming at a contradiction with Kruskal's theorem, assume that $\ti_{\prec}$ fails. This means that the premise $\prog_\prec$ of $\ti_{\prec}$ is true while the conclusion $\forall\alpha\in\varepsilon_0.X\alpha$ is false, for some set~$X$. In other words, we have
\begin{equation*}
\exists\alpha\in\varepsilon_0.\neg X\alpha\quad\text{and}\quad\forall\alpha\in\varepsilon_0(\neg X\alpha\to\exists\beta\prec\alpha.\neg X\beta).
\end{equation*}
We can thus construct a sequence $\alpha_0,\alpha_1,\ldots\subseteq\varepsilon_0$ with $\neg X\alpha_i$ and $\alpha_{i+1}\prec\alpha_i$ for every index~$i\in\mathbb N$. Note that this does not involve choice, as the elements of~$\varepsilon_0$ are coded by natural numbers: in the recursion step, we can thus pick the element~$\alpha_{i+1}$ that is minimal with respect to the usual order~$\leq_{\mathbb N}$ on~$\mathbb N$. In $\aca_0$, the construction can be formalized as follows: Define an $X$-sequence as a finite sequence $\langle\alpha_0,\ldots,\alpha_{n-1}\rangle$ such that for all $i<n$
\begin{enumerate}[label=(\roman*)]
\item we have $\alpha_i\in\varepsilon_0$ and $\neg X\alpha_i$ as well as $\alpha_i\prec\alpha_{i-1}$ or $i=0$,
\item if $\alpha_i'\in\varepsilon_0$ and $\neg X\alpha_i'$ as well as $\alpha_i'\prec\alpha_{i-1}$ or $i=0$, then we have $\alpha_i\leq_{\mathbb N}\alpha_i'$.
\end{enumerate}
Due to the assumption~$\neg\ti_{\prec}$ and the minimality condition~(ii), a straightforward induction in~$\aca_0$ shows that there is a unique $X$-sequence $\langle\alpha_0,\ldots,\alpha_{n-1}\rangle$ of each length~$n\in\mathbb N$ (assuming that $X$-sequences are coded by natural numbers). Given such a sequence, we also note that each initial segment $\langle\alpha_0,\ldots,\alpha_{m-1}\rangle$ is the unique $X$-sequence of length $m<n$. Now an infinite sequence $\alpha_0,\alpha_1,\ldots\subseteq\varepsilon_0$ with  $\neg X\alpha_i$ and $\alpha_{i+1}\prec\alpha_i$ for all~$i\in\mathbb N$ is represented by the set
\begin{equation*}
\{\langle i,\alpha_i\rangle\,|\,\text{``$\alpha_i$ is the $(i+1)$-th entry of some (every) $X$-sequence of lenght~$>i$"}\}.
\end{equation*}
Note that we have put $i+1$ rather than $i$ so that we can call $\alpha_0$ the first rather that  the $0$-th entry. The given set can be formed in $\aca_0$, because the condition is arithmetical (in fact it defines a recursive relation relative to~$X$). Now the function $f:\varepsilon_0\to\mathcal B$ from Proposition~\ref{prop:quasi-embedding} transforms our sequence $\alpha_0,\alpha_1,\ldots\subseteq\varepsilon_0$ into a sequence $f(\alpha_0),f(\alpha_1),\ldots\subseteq\mathcal B$, which can also be formed in $\aca_0$ (use that $\langle\alpha,t\rangle\in f$ is arithmetical with $f$ as parameter). By Kruskal's theorem, we find indices $i<j$ with $f(\alpha_i)\leq_{\mathcal B}f(\alpha_j)$. The latter entails $\alpha_i\preceq\alpha_j$ by Proposition~\ref{prop:quasi-embedding}. On the other hand, the construction of our sequence yields $\alpha_j\prec\alpha_{j-1}\prec\ldots\prec\alpha_i$, so that we have reached the contradiction that we were aiming at.
\end{proof}

Recall, again, that there are two ways to express that $(X,<_X)$ is well founded:
\begin{enumerate}
\item there is no infinite sequence $\alpha_0,\alpha_1,\ldots\subseteq X$ with $\alpha_{i+1}<_X\alpha_i$ for all~$i\in\mathbb N$,
\item any non-empty $Z\subseteq X$ has a $<_X$-minimal element.
\end{enumerate}
Note that~(2) is logically equivalent to transfinite induction along~$<_X$. In a set-theoretic context, the fact that (1) implies~(2) depends on the axiom of dependent choice. We have seen that this is not the case in second order arithmetic, where we can use the usual order on $\mathbb N$ to pick elements without appealing to a choice principle. Indeed, it is implicit in the previous proof that the implication from (1)~to~(2) is provable in~$\aca_0$. For the reader with a background in reverse mathematics (see~\cite{simpson09}), we mention that the same argument establishes this implication in the weaker theory $\mathsf{RCA_0}$ (``recursive comprehension axiom"). The implication from~(2) to~(1) is immediate in set theory: if $\alpha_0,\alpha_1,\ldots$ falsifies~(1), then (2) fails for $Z=\{\alpha_n\,|\,n\in\mathbb N\}$. This last set can be formed in $\aca_0$ but not in~$\mathsf{RCA}_0$. Nevertheless, the implication from~(2) to~(1) is provable in~$\mathsf{RCA}_0$ as well. This is related to a result of computability theory: any infinite recursively enumerable set has an infinite recursive subset (see e.\,g.~\cite[Lemma~2.3.12]{freund-thesis} for a detailed proof in~$\mathsf{RCA}_0$).

We have mentioned that the following result can be derived via Remark~\ref{rmk:eps_0-set-theoretic} or Exercise~\ref{ex:Gentzen-jump}. Here, we argue that it follows from Kruskal's theorem.

\begin{corollary}
The structure $(\varepsilon_0,\prec)$ is a well order.
\end{corollary}
\begin{proof}
For any infinite sequence $\alpha_0,\alpha_1,\ldots$ in $\varepsilon_0$, we can use Kruskal's theorem to find $i<j$ with $\alpha_i\preceq\alpha_j$, as in the previous proof. Hence the given sequence in~$\varepsilon_0$ cannot be strictly decreasing.
\end{proof}

In Exercise~\ref{ex:Gentzen-jump} we have seen that the result of Theorem~\ref{thm:unprov-eps_0} is sharp. In the following remark, we argue that Theorem~\ref{thm:Kruskal-indep-ACA} is sharp as well. Here it is crucial that we consider binary trees only: Kruskal's theorem for arbitrary finite trees reaches far beyond $\aca_0$ and the well order~$\varepsilon_0$ (see~\cite{rathjen-weiermann-kruskal}).

\begin{remark}\label{rmk:reification}
Recall the notion of bad sequence that was introduced in the paragraph before Exercise~\ref{ex:Kruskal}. Let $\bad(\mathcal B)$ be the set of non-empty and (necessarily) finite bad sequences in the partial order~$(\mathcal B,\leq_{\mathcal B})$. By a \emph{reification} we shall mean a function $r:\bad(\mathcal B)\to\varepsilon_0$ such that we have
\begin{equation*}
r(\langle t_0,\ldots,t_n,t_{n+1}\rangle)\prec r(\langle t_0,\ldots,t_n\rangle)\quad\text{when}\quad\langle t_0,\ldots,t_{n+1}\rangle\in\bad(\mathcal B).
\end{equation*}
For the case of~$\mathcal B$, Dick de Jongh has shown that such a reification into $\varepsilon_0$ exists (according to Diana Schmidt~\cite{schmidt75}, cf.~also~\cite{schmidt-habil-new}). Indeed, this fact can be established in~$\aca_0$ (and even in~$\mathsf{RCA}_0$, see~\cite[Section~6]{freund-kruskal-scheme} for a detailed presentation). We can deduce that the following are equivalent over $\aca_0$ (in fact over~$\mathsf{RCA}_0$):
\begin{enumerate}
\item Kruskal's theorem for binary trees (i.\,e., Theorem~\ref{thm:kruskal-binary}),
\item the well foundedness of $(\varepsilon_0,\prec)$.
\end{enumerate}
The implication from~(1) to~(2) has been established in the proof of Theorem~\ref{thm:Kruskal-indep-ACA}. We now sketch a proof of the converse direction: Aiming at a contradiction, assume that~(1) fails. We then have an infinite bad sequence $t_0,t_1,\ldots\subseteq\mathcal B$. Using our reification $r:\mathcal B\to\varepsilon_0$, we get an infinitely descending sequence $r(\langle t_0\rangle)\succ r(\langle t_0,t_1\rangle)\succ\ldots$ in~$\varepsilon_0$, which contradicts~(2). In a sense, the fact that we have an equivalence means that the reification into~$\varepsilon_0$ yields the most elementary proof of Kruskal's theorem. In contrast, the proof via the minimal bad sequence argument (cf.~Exercise~\ref{ex:Kruskal}) is much shorter but much less `elementary': as shown by Alberto Marcone~\cite{marcone-bad-sequence}, a general version of the minimal bad sequence argument exhausts the full strength of a theory known as $\Pi^1_1\textsf{-CA}_0$, which is much stronger than $\aca_0$ (see once again~\cite{simpson09} for general background). Even Kruskal's theorem for arbitrary finite trees is considerably weaker than $\Pi^1_1\textsf{CA}_0$ (see \cite{rathjen-weiermann-kruskal} but compare~\cite{FRW-Kruskal}). In view of Theorem~\ref{thm:Kruskal-indep-ACA}, neither the minimal bad sequence argument nor the proof via~$\varepsilon_0$ can be implemented in~$\aca_0$.
\end{remark}

\section{Conclusion: other applications of ordinal analysis}\label{sect:conclusion}

In these lecture notes, we have focused on one particular application of ordinal analysis: independence for the binary Kruskal theorem. Let us point out that there are considerably stronger results of the same type. In particular, the graph minor theorem is independent of a second order theory known as $\Pi^1_1\textsf{-CA}_0$, as shown by Harvey Friedman, Neil Robertson and Paul Seymour~\cite{friedman-robertson-seymour}. In the present section, we briefly mention two other types of applications. These would have deserved our full attention but have been omitted in order to keep the lecture concise.

The first application concerns \emph{provably total recursive function}. Consider a theory~$\textsf{T}$ in a language that interprets first order arithmetic. Assume we have
\begin{equation*}
\textsf T\vdash\forall x\in\mathbb N\exists y\in\mathbb N.\theta(x,y),
\end{equation*}
where $\{(m,n)\,|\,\mathbb N\vDash\theta(m,n)\}$ is recursively enumerable (officially, we demand that $\theta$ is a $\Sigma^0_1$-formula). Given that $\textsf T$ is sound, the formula it derives must be true. The latter entails that there is a recursive function $f:\mathbb N\to\mathbb N$ such that we have
\begin{equation*}
\mathbb N\vDash\forall x\in\mathbb N\exists y\leq f(x).\theta(x,y).
\end{equation*}
Georg Kreisel has famously asked (cf.~\cite{kreisel-significance}): ``What more do we know if we have proved a theorem by restricted means than if we merely know that it is true?" In the present situation, we can give the following answer: Typically, an ordinal analysis of~$\textsf T$ will yield a bound on a function~$f$ as above. In other words, we obtain quantitative information rather than a mere existence result. Let us point out that such information can also be extracted by other methods, in particular by functional interpretation, which is fundamental for the proof mining programme of Ulrich Kohlenbach~\cite{kohlenbach-proof-mining}. If $\theta$ defines a function (provably in~$\textsf T$), then we can even achieve $y=f(x)$ rather than $y\leq f(x)$ as above. In this case, $f$ is called a provably total recursive function of~$\textsf T$. More information can be found, e.\,g., in~\cite{wainer-fairtlough-98}.

As a second application, we mention relative consistency and conservativity. In Proposition~\ref{prop:aca-conservative} we have seen that the second order theory~$\aca_0$ is conservative over first order Peano arithmetic. We have established this proposition by a model-theoretic argument. In other important cases, this type of argument is not available, but ordinal analysis can be used to prove conservativity. As an example, we mention a result due to Michael Rathjen~\cite{rathjen_martin-loef} and Anton Setzer~\cite{setzer_martin-loef}: a strong $\mathcal L_2$-theory known as $\Delta^1_2\textsf{-CA+BI}$ is $\Pi^0_2$-conservative over Per Martin-L\"of's type theory (in the 1984 version). An important point here is that the $\mathcal L_2$-theory is based on classical logic while the type theory is constructive. In a sense, the conservativity result tells us that we may freely use classical logic even if we think that constructive reasoning is preferable. It can be seen as a partial realization of Hilbert's programme, which had the aim to justify abstract reasoning by showing that it can be eliminated in principle (see again~\cite{zach16,rathjen-sieg-stanford}). Indeed, Hilbert's programme and the question of consistency form the historical root of ordinal analysis. Out of this root, ordinal analysis has grown into a powerful method with a large variety of applications. For a more comprehensive account of the various aspects of ordinal analysis, the reader may wish to consult the survey~\cite{rathjen-realm} by Michael Rathjen, as well as the other sources that are mentioned in the introduction to these lecture notes.

\bibliographystyle{amsplain}
\bibliography{Lecture_Unprovability}

\end{document}